\documentclass[journal,10pt,onecolumn,draftclsnofoot]{IEEEtran}

\ifCLASSINFOpdf
\else
   \usepackage[dvips]{graphicx}
\fi
\usepackage{url}

\usepackage[caption=false,font=normalsize,labelfont=sf,textfont=sf]{subfig}
\usepackage{graphicx}
\usepackage{amsmath,bbm,comment,hyperref,amsthm,hyperref} 
\usepackage[a4paper, margin=1in]{geometry}
\newtheorem{theorem}{Theorem}
\newtheorem{corollary}[theorem]{Corollary}
\newtheorem{lemma}[theorem]{Lemma}
\newtheorem{remark}{Remark}
\newtheorem{assumption}{Assumption}
\usepackage{enumitem} 
\newtheorem{definition}{Definition}

\usepackage{xcolor}
\newcommand{\bY}{\mathbbm{Y}}
\newcommand{\bR}{\mathbbm{R}}

\newcommand{\cP}{\mathcal{P}}

\newcommand{\cM}{\mathcal{M}}

\newcommand{\Exp}{{\rm{E}}}
\newcommand{\Var}{\textup{Var}}
\newcommand{\bP}{{\rm{P}}}

\begin{document}

\title{Optimal Decision Rules for Composite Binary Hypothesis Testing under Neyman-Pearson Framework}

\author{Yanglei Song, Berkan Dulek, and Sinan Gezici, \emph{Fellow, IEEE}
\thanks{Y. Song is with the Department of Mathematics and Statistics, Queen's University, Kingston, Ontario, Canada (e-mail: yanglei.song@queensu.ca). B. Dulek is with the Department of Electrical and Electronics Engineering,
Hacettepe University, Beytepe Campus, Ankara 06800, Turkey (e-mail:
berkan@ee.hacettepe.edu.tr). S. Gezici is with the Department of Electrical and Electronics Engineering, Bilkent University, Ankara 06800, Turkey, (e-mails: gezici@ee.bilkent.edu.tr). Y. Song acknowledges the support by NSERC Grant RGPIN-2020-04256. B. Dulek was supported by the Scientific and Technological Research Council of Turkey (TUBITAK) under the Grant Number 122E493. S. Gezici was supported by TUBITAK under the BIDEB 2219 Program. B. Dulek and S. Gezici thank TUBITAK for their supports.}}


\maketitle

\begin{abstract}
The composite binary hypothesis testing problem within the Neyman-Pearson framework is considered. The goal is to maximize the expectation of a nonlinear function of the detection probability, integrated with respect to a given probability measure, subject to a false-alarm constraint. It is shown that each power function can be realized by a generalized Bayes rule that maximizes an integrated rejection probability with respect to a finite signed measure. For a simple null hypothesis and a composite alternative, optimal single-threshold decision rules based on an appropriately weighted likelihood ratio are derived. The analysis is extended to composite null hypotheses, including both average and worst-case false-alarm constraints, resulting in modified optimal threshold rules. Special cases involving exponential family distributions and numerical examples are provided to illustrate the theoretical results.
\end{abstract}


\IEEEpeerreviewmaketitle

\newpage 

\section{Introduction}\label{sec:Intro}

In binary hypothesis testing, the goal is to decide between two possible statistical scenarios, known as \emph{hypotheses}, based on an observation \cite{NP1933}--\nocite{Poor,romano2005testing}\cite{KayDetection}. This framework arises in various applications such as target detection, data communication, and classification \cite{richards2005fundamentals}--\nocite{trees2001detection}\cite{DetectionRadar}. Depending on whether the observation under each hypothesis corresponds to a single probability distribution or a family of distributions, hypothesis testing problems are categorized as \emph{simple} or \emph{composite}. 

In \emph{simple} binary hypothesis testing problems, the observation follows a single distribution under each hypothesis. For these problems, optimal decision rules can readily be derived using classical optimality frameworks such as Bayesian, minimax, and Neyman-Pearson (NP) hypothesis testing \cite{NP1933}--\nocite{Poor,romano2005testing}\cite{KayDetection}. In particular, \emph{likelihood ratio tests} are shown to be optimal, where the ratio of the distributions under the two hypotheses is compared to a threshold. In contrast, \emph{composite} binary hypothesis testing---where at least one hypothesis is composite---varies in complexity depending on the chosen optimality criterion. For example, in the Bayesian framework, these problems are as tractable as simple hypothesis testing by using averaged likelihood functions in a likelihood ratio test \cite{Poor}. However, the problem can become complicated in the NP framework, which is the main focus of this paper.

In the NP framework for distinguishing between two \emph{simple} hypotheses, the goal is to maximize the detection probability while maintaining a constraint on the false-alarm probability \cite{Poor, trees2001detection}. When the null hypothesis is composite, the common practice is to enforce the false-alarm constraint for all possible distributions under the null hypothesis \cite{Poor}, \cite{cvitanic2001generalized,rudloff2008testing}. In contrast, when the alternative hypothesis is composite, various approaches are employed in the literature. One approach is to seek a uniformly most powerful (UMP) decision rule that maximizes the detection probability for all possible distributions under the alternative hypothesis, while satisfying the false-alarm constraint \cite{Poor, trees2001detection}. However, such a decision rule exists only under specific conditions \cite{Poor}. An alternative approach is to maximize the average detection probability under the false-alarm constraint \cite{bayarri2004interplay}--\nocite{lehmann2009history,begum2007new}\cite{berger1997unified}. In this case, the problem reduces to an NP problem for a simple alternative hypothesis by representing the alternative distribution as the integrated distribution over the prior on the parameter under the alternative hypothesis \cite{bayram2011restricted}.  
This approach, however, requires a prior distribution to compute the average detection probability. If such a prior distribution is unavailable (as is often the case in NP formulations), a maximin approach can be used. This method aims to maximize the minimum detection probability subject to the false-alarm constraint \cite{cvitanic2001generalized,rudloff2008testing}. The solution can be characterized by an NP decision rule corresponding to a least-favorable distribution under the alternative hypothesis in a dual problem formulation \cite{rudloff2008testing}. Notably, considering the least-favorable distribution is equivalent to assuming a worst-case scenario, which corresponds to a conservative approach. Modifications to this approach are studied in \cite{huber1973minimax,augustin2002neyman} utilizing the interval probability concept. As a generalization, the restricted NP approach is proposed in \cite{bayram2011restricted}, addressing uncertainty in the prior distribution of the alternative hypothesis. A restricted NP decision rule aims to maximize the average detection probability calculated based on the uncertain prior distribution while ensuring that the minimum detection probability does not fall below a predefined threshold and that the false-alarm probability remains below a specified significance level. 
It is shown that the restricted NP decision rule can be specified as a classical NP decision rule corresponding to a certain least-favorable distribution \cite{bayram2011restricted}. 

The generalized likelihood ratio test (GLRT) is also a common and practical approach for composite binary hypothesis testing \cite{Poor}, \cite{trees2001detection}. In the GRLT, the maximum likelihood estimates of the unknown parameter are obtained under both hypotheses using the observation and substituted into the corresponding likelihood expressions in place of the unknown parameter value. In this way, a likelihood ratio test is formed based on the parameter values that best fit the observation under both hypotheses. While the GLRT is easy to design and implement, it lacks optimality properties in general for finitely many observations, though it has some asymptotic optimality properties \cite{zeitouni2002generalized}. As an alternative asymptotic approach, the studies in \cite{hoeffding1965asymptotically}--\nocite{natarajan1985large,zeitouni1991universal}\cite{levitan2002competitive} adopt a generalized version of the NP criterion, which aims to maximize the misdetection exponent uniformly across all distributions under the alternative hypothesis subject to the false-alarm exponent constraint. Moreover, in \cite{ITW2022}, a simple null and composite alternative hypothesis testing problem is considered under the NP framework, and a decision rule is proposed to achieve the optimal error exponent tradeoff for each alternative distribution.

Our work addresses \textit{non-asymptotic} composite binary hypothesis testing problems within an NP detection theoretic framework. Our primary motivation stems from recent advances in behavioral utility-based detection problems \cite{Gezici_PHT}--\nocite{geng20a,geng20b}\cite{dulek24}. Specifically, for decision-making tasks involving humans in the loop, their cognitive biases and subjective perception of probabilities, gains, and losses need to be taken into account. These effects can be incorporated into the system performance metric using prospect-theoretic probability weighting and cost valuation functions \cite{ProsHT}, which introduce a nonlinear relationship between actual and perceived performance metrics. Despite extensive prior research, deriving optimal decision rules under such \textit{nonlinearities} remains a challenge. In this regard, the main contributions and novel aspects of this manuscript can be listed as follows.

\begin{enumerate} 
\item In Section \ref{subsec:structural_results}, we derive structural results on the set of all power functions, where the power function of a randomized decision rule is defined as its rejection probabilities over the parameter space. Specifically, we show that each power function can be realized by a ``generalized Bayes rule'' that maximizes the integral of a power function with respect to (w.r.t.) some finite \textit{signed} measure (see Theorem \ref{thrm:general_case}). Thus, for \textit{any} hypothesis testing problem whose formulation solely depends on power functions, it suffices to consider the family of ``generalized Bayes rules''.

\item In Section \ref{sec:simple_null_comp_alter}, we consider the problem of testing a simple null hypothesis against a composite alternative. The objective is to maximize the expected value of a continuously differentiable but otherwise arbitrary function $g(\cdot)$ of the detection probability, subject to a constraint on the false-alarm probability. The expectation is taken w.r.t.~a given probability measure over the unknown parameter under the alternative hypothesis. We show that optimal detection is achieved using a single-threshold test on the following statistic: the ratio of a probability density function (p.d.f.), averaged over the alternative hypothesis, to the p.d.f.~under the null hypothesis, where the averaging measure depends on the power function of the optimal detection rule (see Theorem \ref{theorem:simple_null_structrual}).
If $g$ is strictly increasing and the likelihood ratio can be expressed as a convex function of a scalar-valued sufficient statistic for the family of distributions indexed by the unknown parameter, then the optimal test accepts the null hypothesis when the sufficient statistic falls within an interval  (possibly semi-infinite) of the real line and rejects otherwise. 
    
\item In Sections \ref{sec:composite_null_integrated} and \ref{sec:composite_null_supremum}, we extend our analysis to the case where the null hypothesis is also composite, while maintaining the same objective function as in Section \ref{sec:simple_null_comp_alter}.  In Section \ref{sec:composite_null_integrated}, we impose a constraint on the average false-alarm probability w.r.t.~a given probability measure $\Lambda_0$ for the unknown parameter under the null hypothesis. In contrast, in Section \ref{sec:composite_null_supremum}, the constraint is on the worst-case false-alarm probability under the null. 

For average false-alarm control, we show that if we replace the single null p.d.f.~from the previous part with the $\Lambda_0$-integrated p.d.f.~under the composite null hypothesis, then the modified single-threshold rule is optimal (see Theorem \ref{thrm:composite_null_avg}). In addition, for the supremum false-alarm control, we establish that replacing the single null p.d.f.~with the integrated p.d.f., taken w.r.t.~some least-favorable distribution under the null, yields the optimal decision rule (see Theorem \ref{theorem:composiste_supremum}).

    \item Throughout the paper, several theorems and corollaries are derived for the special case of single-parameter exponential family distributions (see Corollary \ref{cor:exponential_family}, Theorem \ref{theorem:composiste_null_exponential} and Theorem \ref{theorem:composiste_exponential_supremum}). Numerical examples from behavioral utility-based detection theory are also provided to corroborate the theoretical results.   
\end{enumerate}

\section{Problem Formulation}\label{sec:Pre}

Consider a compact parameter space $\Theta \subset \bR^{d_1}$ for some integer $d_1 \geq 1$, and partition it into two disjoint, non-empty subsets, $\Theta_0$ and $\Theta_1$, such that $\Theta = \Theta_0 \cup \Theta_1$ and  $\Theta_0 \cap \Theta_1 = \emptyset$. Denote by $Y$ an observation that takes values in the observation space $\bY \subset \bR^{d_2}$ for some integer $d_2 \geq 1$. Let $\mu$ be a $\sigma$-finite measure on $\mathbbm{Y}$, and for each $\theta \in \Theta$, let $f_{\theta}$ be a p.d.f. w.r.t. $\mu$. When $\theta \in \Theta$ is true, the observation $Y$ has a density $f_{\theta}$ w.r.t.~$\mu$.


The goal is to decide, based on the observation $Y$, whether $\theta$ belongs to $\Theta_0$ or $\Theta_1$, which correspond to hypothesis ${\mathcal{H}}_0$ and hypothesis ${\mathcal{H}}_1$, respectively. Specifically, a randomized decision rule $\delta$ is a measurable function from $\mathbbm{Y}$ to $[0,1]$ such that $\delta(Y)$ represents the probability of rejecting ${\mathcal{H}}_0$ (or equivalently, accepting ${\mathcal{H}}_1$). Let $\Delta$ denote the collection of all such randomized decision rules. 
For each rule $\delta \in \Delta$ and  $\theta \in \Theta$, we define
\begin{equation*}
p(\theta;\delta) := {\rm{E}}_{\theta}[\delta(Y)] = \int_\mathbbm{Y} \delta(y) f_{\theta}(y) \mu(dy),
\end{equation*}
which is the probability of rejecting ${\mathcal{H}}_0$ when $\theta$ is the true parameter value. We refer to $p(\cdot;\delta) $ as the \emph{power function} of $\delta$, and let 
\begin{equation}\nonumber
\mathcal{P} = \{p(\cdot\,; \delta): \delta \in \Delta\}
\end{equation}
denote the set of all power functions that can be realized. 

In this work, we investigate  various optimization problems with the aim of finding decision rules that optimize a certain objective function subject to constraints, where both the objective function and constraints depend solely on the associated power functions.

\section{Structural Results}\label{subsec:structural_results}

In this section, we establish structural results for the set $\mathcal{P}$ of all power functions, and show that for any hypothesis testing problem, it is sufficient to focus on the family of ``generalized Bayes rules''.


For each $\theta \in \Theta$, denote by $\mu_{\theta}$ the distribution of the observation $Y$ when $\theta$ is the true parameter value, that is, 
$d \mu_{\theta}/d\mu(y) = f_{\theta}(y)$ for $y \in \bY$, where $d \mu_{\theta}/d\mu$ denotes the Radon-Nikodym derivative of $\mu_{\theta}$ w.r.t.~$\mu$.

\begin{assumption}
        \label{assumption:continuity}
The function $\theta \mapsto \mu_{\theta}$ is continuous in total variation distance, that is,  
\begin{equation*}
   \lim_{n\to \infty} \int_{\bY} |f_{\theta_n}(y) - f_{\theta}(y)| \mu(dy) = 0,
 \end{equation*}
 for  all $\theta_n, n \geq 1$, $\theta \in \Theta$  and  $\lim_{n \to \infty} \theta_n = \theta$. Further, for any $\theta, \theta' \in \Theta$, $\mu_{\theta}$ and $\mu_{\theta'}$ are mutually absolutely continuous.
 \end{assumption}

The last condition of Assumption \ref{assumption:continuity} requires that $\{\mu_{\theta}: \theta \in \Theta\}$ has the same support. Denote by $C(\Theta)$ the space of continuous functions on the compact set $\Theta$, and we equip it, as well as its subsets, with the supremum norm $\|\cdot\|_{\infty}$. For a vector $v \in \bR^{d_1}$, denote its Euclidean norm by $\|v\|$. We write  ``$\mu$-a.e.'' to mean  ``$\mu$-almost everywhere'' for simplicity. We summarize some properties about $\cP$ in the following lemma and provide its proof in Appendix \ref{app:proof_cP}.

\begin{lemma}\label{lemma:structural}
Suppose Assumption \ref{assumption:continuity} holds. 
\begin{enumerate}
    \item $\mathcal{P} \subset C(\Theta)$.
    \item $\mathcal{P}$ is convex and compact.
    \item Assume that $\Theta \subset \bR^{d_1}$ has a non-empty interior $\Theta^{\circ}$.  Let $U \subset \Theta^{\circ}$ be a non-empty open set.  Assume that for $\mu$-a.e.~$y$, the function $\theta \in U \mapsto f_{\theta}(y) \in \bR$ is differentiable with gradient $\dot f_{\theta}(y)$. In addition,  there exists  a function $F:\bY \to \bR$ such that  $\sup_{\theta \in U}\|\dot f_{{\theta}}(y)\| \leq F(y)$ for $\mu$-a.e.~$y$ and $\int_{\bY} F(y)\mu(dy) < \infty$. Then the interior of $\cP$ is empty.  
    \item Let $p \in \cP$. If $0 < p(\theta') < 1$ for some $\theta' \in \Theta$, then $0 < \inf_{\theta \in \Theta} p(\theta) \leq \sup_{\theta \in \Theta} p(\theta) < 1$. 
\end{enumerate}
\end{lemma}

\begin{remark}
In part 3) of Lemma \ref{lemma:structural}, we require the parameter space $\Theta$ to have a non-empty interior (i.e., containing an open set), and thus exclude the case of finite $\Theta$ for that part. The intuition is that, under the assumptions in part 3), any $p \in \cP$ is a differentiable function on $U$. However, $C(\Theta)$ is sufficiently rich that in any neighborhood of $p \in \cP$, there exist continuous but non-differentiable functions on $U$.
\end{remark}

\subsection{Generalized Bayes Rules}

Denote by $\mathcal{M}(\Theta)$, $\mathcal{M}_{+}(\Theta)$ and $\mathcal{M}_1(\Theta)$ the space of finite signed measures, finite measures, and probability measures on $\Theta$, respectively (see \cite[Section 3.1]{folland1999real} for definitions of signed measures). Further, recall the Hahn-Jordan decomposition of signed measures in \cite[Theorem 3.4]{folland1999real}.

For two probability measures $\pi^{+}, \pi^{-} \in \mathcal{M}_1(\Theta)$, a non-negative real number $c \geq 0$, and a measurable function $\gamma:\bY \to [0,1]$, we define the following rule: for $y \in \bY$, 
\begin{align}
    \label{def:optimal_rule}
    \delta^*(y; \pi^{+}, \pi^{-}, c, \gamma)=
\begin{cases}
1, & \text{if } \int_{\Theta} f_{\theta}(y) \pi^+(d\theta) 
> c \int_{\Theta} f_{\theta}(y) \pi^-(d\theta)\\
\gamma(y), & \text{if } \int_{\Theta} f_{\theta}(y) \pi^+(d\theta) 
= c \int_{\Theta} f_{\theta}(y) \pi^-(d\theta)\\
0, & \text{if } \int_{\Theta} f_{\theta}(y) \pi^+(d\theta) 
< c \int_{\Theta} f_{\theta}(y) \pi^-(d\theta)
\end{cases}.
\end{align}

The next lemma shows that when the objective is to maximize the averaged power function with respective to some finite \textit{signed} measure $\nu \in \cM(\Theta)$, the optimal rule must be of the form in \eqref{def:optimal_rule}. For this reason, we refer to this family of decision rules as ``generalized Bayes rules''. (Conventional Bayes rules employ positive measures and do not focus on the equality case as it does affect the Bayes risk.)

\begin{lemma}\label{lemma:Bayes_rule}
 Let $\nu \in \cM(\Theta)$ be a finite \textit{signed} measure. Denote by $\nu^{+} \in \cM_{+}(\Theta)$ and $\nu^{-} \in  \cM_{+}(\Theta)$  the positive and negative parts in the Hanh decomposition of $\nu$, respectively. Assume $\nu^{+}$ and $\nu^{-}$ are non-zero. Consider the optimization problem:
\begin{align}
    \label{def:Bayes_V}
\max_{\delta \in \Delta}\; \int_{\Theta} p(\theta;\delta) \nu(d\theta).
\end{align}
Let $\delta \in \Delta$. Then $\delta$ is an optimal rule for \eqref{def:Bayes_V} if and only if $\delta$ is equal to, up to $\mu$-a.e.-$y$,  
$\delta^*(\cdot;\pi^{+}, \pi^{-}, c, \gamma)$, where
 \begin{align*}
& \pi^+ := \frac{\nu^+}{\int_{\Theta} \nu^+(d\theta)},\quad  \pi^- := \frac{\nu^-}{\int_{\Theta} \nu^-(d\theta)}, \\
& c := {\int_{\Theta} \nu^-(d\theta)}{\Big/}{\int_{\Theta} \nu^+(d\theta)},
 \end{align*}
and $\gamma:\bY \to [0,1]$ is some measurable function.
\end{lemma}
\begin{proof}
    See Appendix \ref{app:proof_cP}.
\end{proof}

\begin{remark}\label{remark:zero}
If $\nu^{-}$ (resp.~$\nu^{+}$) is the zero measure, then the optimal rule above is given by $\tilde{\delta}(y) = 0$ (resp.~$\tilde{\delta}(y) = 1$) for $y \in \bY$. In either case, the optimal rule is in the form of \eqref{def:optimal_rule}.
\end{remark}

\begin{remark}
Rules in the form of  \eqref{def:optimal_rule} may also be viewed as generalizations of NP rules for simple versus simple hypothesis testing problems. In particular, $\pi^{+}$ and $\pi^{-}$ are two probability measures on $\Theta$ with \emph{disjoint support}, and we threshold the ratio of the two integrated likelihood functions.
\end{remark}

\subsection{General Structural Results}\label{subsec:general_cases}

In this subsection, we focus on the case that \emph{$\cP$ has an empty interior}.  As noted in Lemma \ref{lemma:structural}, if the parameter space $\Theta$ has a non-empty interior (e.g., if $\Theta \subset \bR$ contains an interval), under mild technical conditions, $\cP \subset C(\Theta)$ indeed has an empty interior.

The following theorem shows that  any power function can be realized by some ``generalized Bayes rule'' in \eqref{def:optimal_rule}.

\begin{theorem}\label{thrm:general_case}
Let $p \in \cP$ be an arbitrary power function. Suppose that Assumption \ref{assumption:continuity} holds, and that $\cP$ has an empty interior. 
There exist some $\pi^{+}, \pi^{-} \in \mathcal{M}_1(\Theta)$,  $c \geq 0$, and $\gamma:\bY \to [0,1]$ such that  ${\pi}^+$ and ${\pi}^{-}$ are mutually singular,  and $p$ is the power function associated with the decision rule $\delta^*(\cdot; \pi^{+}, \pi^{-}, c, \gamma)$.  
\end{theorem}

\begin{proof}
    See Appendix \ref{app:proof_general_case}.
\end{proof}

The above theorem establishes that for any hypothesis testing problem, if its formulation depends solely on the power function, it is sufficient to consider the family of generalized Bayes rules in \eqref{def:optimal_rule}. However, at this generality, we cannot explicitly characterize the support of $\pi^+$ or $\pi^{-}$ for each power function $p \in \cP$. In subsequent sections, we analyze specific optimization problems and derive additional results concerning the structure of optimal decision rules.

\begin{remark}
The key strategy for proving Theorem \ref{thrm:general_case} is to establish that every $p \in \cP$ is a support point of $\cP$, that is, it solves the optimization problem in \eqref{def:Bayes_V} corresponding to some non-zero signed measure $\nu \in M(\Theta)$. This enables us to apply Lemma \ref{lemma:Bayes_rule}.

Note that for general Banach spaces, such as the space of absolutely summable sequences, not every point in a closed, convex subset with an empty interior is a support point; see Example 7.8 in \cite{aliprantis2006infinite}.
\end{remark}

For any $\pi^{+}, \pi^{-} \in \mathcal{M}_1(\Theta)$ and  $c \geq 0$, we define $\mathcal{B}(\pi^{+}, \pi^{-},c)$ to be the following set:
\begin{align*}
 \mathcal{B}(\pi^{+}, \pi^{-},c)=
 \left\{y \in \mathbbm{Y} :
    \int_{\Theta} f_{\theta}(y) \pi^+(d\theta) 
= c\int_{\Theta} f_{\theta}(y) \pi^-(d\theta)
    \right\}.
\end{align*}
If such sets have zero probabilities for \emph{mutually singular} $\pi^{+}$ and $\pi^{-}$, then it suffices to consider those \textit{deterministic} (i.e., without randomization) generalized Bayes rules in \eqref{def:optimal_rule}, which follows immediately from Theorem \ref{thrm:general_case}. 

\begin{corollary}
    \label{cor:general_case}
    Let $p \in \cP$ be an arbitrary power function. 
Assume the conditions in Theorem \ref{thrm:general_case} hold. Further, assume for all $\pi^{+}, \pi^{-} \in \mathcal{M}_1(\Theta)$ that are mutually singular and  any $c \geq 0$,
\begin{equation}
   \label{def:zero_prob}
\mu\left(\mathcal{B}(\pi^{+}, \pi^{-},c)\right) =0.
\end{equation}
Let $\bar{\gamma}_1(y)=1$ for $y \in \bY$. 
There exist some $\pi^{+}, \pi^{-} \in \mathcal{M}_1(\Theta)$ and  $c \geq 0$ such that  ${\pi}^+$ and ${\pi}^{-}$ are mutually singular,  and $p$ is the power function associated with the decision rule $\delta^*(\cdot; \pi^{+}, \pi^{-}, c, \bar{\gamma}_1)$.  
\end{corollary}

\begin{remark}
When $\mu$ is the Lebesgue measure on $\bR^{d_2}$ and $\{f_{\theta}: \theta \in \Theta\}$ is the family of exponential family distributions in \eqref{def:exponential_family}, the ``zero probability'' condition in \eqref{def:zero_prob} holds.
\end{remark}

\begin{remark}
By part 3) of Lemma \ref{lemma:structural}, for composite hypotheses, $\cP$ typically has an empty interior. However, when $\cP$ has a non-empty interior --- for instance, when $\Theta$ has finitely many elements --- we can apply the usual supporting hyperlane theorem \cite[Lemma 7.7]{aliprantis2006infinite} to conclude that every boundary point of $\cP$ is a support point. Consequently, in this case,   every power function $p \in \cP$ can be realized as a convex combination of at most two generalized Bayes rules of the form in  \eqref{def:optimal_rule}.
\end{remark}

\section{Simple Null versus Composite Alternative}\label{sec:simple_null_comp_alter}
Let $\theta_0 \in \Theta$. In this section, we aim to test
\begin{align*}
{\mathcal{H}}_0: \theta = \theta_0, \text{ vs }\, {\mathcal{H}}_1: \theta \in \Theta_1 := \Theta \setminus \{\theta_0\}.
\end{align*}
For instance, if $\Theta = [-K,K]$ for some $K > 0$ and $\theta_0 = 0$, then the test is two-sided.  In general, however, there is no uniformly most powerful (UMP) test \cite{romano2005testing}. As an example, if 
$f_{\theta}$ represents the density of the normal distribution with mean $\theta$ and variance 1, no UMP test exists. Our goal is to solve the following problem
\begin{align} \label{def:two_sided_prob}
\begin{split}
&\max_{\delta \in \Delta} \int_{\Theta_1} g(p(\theta;\delta)) \Lambda_1(d\theta), \\
&\text{subject to } p(\theta_0;\delta) \leq \alpha,  
\end{split}
\end{align}
where $\alpha \in (0,1)$ is a user-specified level,  $g:[0,1] \to \bR$ is a measurable function, and $\Lambda_1(\cdot)$ is a probability measure on $\Theta_1$. Note that the weighting probability $\Lambda_1(\cdot)$ is general, e.g., it can be finitely discrete.

\begin{assumption} \label{assumption_g_differentiable}
  $g$ is continuously differentiable on $(0,1)$ and continuous at $\{0,1\}$.
\end{assumption}

 If $g(t) = t$ for $t \in [0,1]$, the optimal test follows from the NP lemma. The problem becomes more interesting when $g$ is nonlinear, such as $g(t) = t^\kappa$ for $ t \in [0,1]$, where $\kappa > 0$ is given.  Another example, inspired by the literature on prospect theory-based hypothesis testing \cite{Gezici_PHT,ProsHT}, is $g(t) = \omega_{v}(t)$ 
 for $t \in [0,1]$, where $v > 0$ is user-specified, and
\begin{equation}
    \label{def:w_v}
\omega_{v}(t) := \frac{t^{v}}{\left(t^v +(1-t)^v\right)^{1/v}}.
\end{equation}
To derive stronger results, we impose the following assumption \textit{in certain cases}, which requires $g$ to be strictly increasing. This assumption is satisfied by the examples mentioned above.

 \begin{assumption}
     \label{assumption_g_increasing}
    $g'(t) > 0$ for $0 < t < 1$.
 \end{assumption}

 Below, we denote by $\bar{\delta}_0$ the decision rule that always accepts $\mathcal{H}_0$, that is, $\bar{\delta}_0(y) = 0$ for $y \in \bY$.




\begin{theorem}\label{theorem:simple_null_structrual}
Suppose that Assumptions \ref{assumption:continuity} and \ref{assumption_g_differentiable} hold, and that 
$\bar{\delta}_0$ is not an optimal solution to problem \eqref{def:two_sided_prob}. 
\begin{enumerate}[label=\roman*)]
    \item 
There exists an optimal rule $\delta^* \in \Delta$ for problem \eqref{def:two_sided_prob} such that for  some constant $\kappa \in \bR$ and a measurable function $\gamma:\bY \to [0,1]$,
 \begin{align}\label{def: optimal_rule_simple_null}
  \delta^*(y) = \begin{cases}
      1, \;&\text{ if }  \;H(y) > \kappa {f_{\theta_0}(y)} \\
      \gamma(y), \;&\text{ if }  \;H(y) = \kappa {f_{\theta_0}(y)}\\
      0, \;&\text{ if }  \;H(y) < \kappa{f_{\theta_0}(y)}
  \end{cases},
 \end{align}
for each $y \in \bY$, where 
\begin{align}
    \label{def:h_function}
    H(y) :=\int_{\Theta_1} g'(p(\theta; \delta^*)) {f_{\theta}(y)} \Lambda_1(d\theta).
\end{align}

\item For any optimal rule $\delta^* \in \Delta$, it must take the form given in \eqref{def: optimal_rule_simple_null}, with the function $H(\cdot)$ defined in \eqref{def:h_function}, for $\mu$-a.e.~$y$.

\item If, in addition to Assumptions \ref{assumption:continuity} and \ref{assumption_g_differentiable}, Assumption \ref{assumption_g_increasing} holds, then $\bar{\delta}_0$ is not an optimal solution to problem \eqref{def:two_sided_prob}, and for any optimal rule $\delta^*\in \Delta$, we have $p(\theta_0; \delta^*) = \alpha$.
\end{enumerate}
\end{theorem}
\begin{proof}
See Appendix \ref{app:proof_two_sided_theorem}.
\end{proof}

\begin{remark}  The proof strategy for Theorem \ref{theorem:simple_null_structrual} is to establish that $\delta^*$ is the optimal decision rule for a \emph{simple versus simple} hypothesis testing problem by linearizing the objective function in \eqref{def:two_sided_prob} at $p(\cdot;\delta^*) \in \cP$.
\end{remark}

\begin{remark}
In general, such as when $g$ is a decreasing function,  $\bar{\delta}_0$ may be an optimal solution to problem \eqref{def:two_sided_prob}. 
In Theorem \ref{theorem:simple_null_structrual}, we exclude this case, which implies that for any optimal rule $\delta^*$,  $p(\theta; \delta^*) \in (0,1)$ for all $\theta \in \Theta$, and thus $g'(p(\theta; \delta^*) )$ is well defined.
\end{remark}

\begin{remark}\label{remark:1_2}
Assumption \ref{assumption:continuity} holds under mild conditions; for example, it holds if for $\mu$-a.e.~$y$, $\theta \in \Theta \mapsto f_{\theta}(y) \in [0,\infty)$ is continuous, and $\sup_{\theta \in \Theta}|f_{\theta}(y)| \leq F(y)$ where $\int F(y) \mu(dy) < \infty$. 
Assumptions \ref{assumption_g_differentiable} and \ref{assumption_g_increasing} concern the function $g$ in \eqref{def:two_sided_prob}. 
\end{remark}

The preceding theorem establishes that an optimal rule exists, which rejects (resp.~accept) $\mathcal{H}_0$ if the function $H(\cdot)$ is strictly above (resp.~below) a multiple of $f_{\theta_0}$. Further, any optimal rule must be of this form. It is clear that the optimal rule in \eqref{def: optimal_rule_simple_null} is a special case of \eqref{def:optimal_rule}, where $\pi^{-}$ is supported on the singleton $\{\theta_0\}$, and if $g$ is strictly increasing, $\pi^{+}$ is supported on $\Theta_1$.

The above structural results can significantly be simplified under additional assumptions.

\begin{corollary}\label{cor:convexity}
Suppose that Assumptions \ref{assumption:continuity}, \ref{assumption_g_differentiable}, and \ref{assumption_g_increasing} hold. 
Let $T: \mathbbm{Y} \to \mathbbm{R}$ be a measurable function, and assume that for each $\theta \in \Theta_1$, there exists a \emph{strictly} convex function $\phi_{\theta}:\mathbbm{R} \to \mathbbm{R}$ such that
    \begin{align}
        \label{assumption:convex_suff}
    \frac{f_{\theta}(y)}{f_{\theta_0}(y)} = \phi_{\theta}(T(y)), \;\;\text{ for } y \in \mathbbm{Y}.
        \end{align}
Then, there exists an optimal rule $\delta^*$ for problem \eqref{def:two_sided_prob} such that $p(\theta_0; \delta^*) = \alpha$, and for some constants $-\infty \leq \ell < u \leq \infty$ and a measurable function $\gamma:\bY \to [0,1]$, we have that for each $y \in \bY$,
 \begin{align} \label{def:optimal_rule_Ty}
  \delta^*(y) = \begin{cases}
      1, \;&\text{ if }  \;T(y) > u \text{ or } T(y) < \ell\\
      \gamma(y), \;&\text{ if }  \;T(y) \in \{\ell, u\} \\
      0, \;&\text{ if }  \;\ell < T(y) < u
  \end{cases}.
 \end{align}
 Further, assume $\mu(\{y \in \bY: T(y) = c\}) = 0$ for $c \in \bR$. Then, for some constants $-\infty \leq \ell < u \leq \infty$,
\begin{equation*}
\delta^*(y) = \mathbbm{1}\{T(y) \geq u \text{ or }  \ T(y) \leq \ell\}.
\end{equation*}
\end{corollary}
\begin{proof}
See Appendix \ref{app:proof_corollaries}.
\end{proof}

\begin{remark}
The function $T(y)$ in \eqref{assumption:convex_suff} is a sufficient statistic for the family $\{f_{\theta}(\cdot); \theta \in \Theta\}$.
\end{remark}

\subsection{Exponential Family Distributions}

In this subsection, we illustrate the preceding results with examples from the single-parameter exponential family distributions \cite[Chapter~4.2]{van2000asymptotic}. Specifically, let
\begin{equation}
    \label{def:exponential_family}
f_{\theta}(y) = c(\theta) h(y) e^{\theta T(y)},
\end{equation}
where $h: \bY \to [0,\infty)$ and $T:\bY \to \bR$ are given, and the domain is
$$
\mathcal{D} := \{\theta  \in \bR: c(\theta)^{-1} :=  \int_{\bY} h(y) e^{\theta T(y)} \mu(dy) < \infty\}.
$$
Assume that $\mathcal{D}$ is open and that $\Theta \subset \mathcal{D}$ is compact.

\begin{corollary}\label{cor:exponential_family}
For the exponential family distributions in \eqref{def:exponential_family},
 Assumption \ref{assumption:continuity} 
and condition \eqref{assumption:convex_suff} hold.
\end{corollary}
\begin{proof}
    See Appendix \ref{app:proof_corollaries}.
\end{proof}

If Assumptions \ref{assumption_g_differentiable} and \ref{assumption_g_increasing} hold, which concern the function $g$ in \eqref{def:two_sided_prob}, then Corollary \ref{cor:convexity} can be applied to the above exponential family distributions. More concretely, let $\mu$ be the Lebesgue measure and $\bY = \bR$. 
Corollary \ref{cor:convexity} applies to the following families $\{f_{\theta}: \theta \in \Theta\}$:
\begin{itemize}
    \item Normal (Gaussian) densities with a known variance $\sigma^2 > 0$: $f_{\theta}(y) = \frac{1}{\sqrt{2\pi} \sigma
    }\exp\left(-\frac{1}{2\sigma^2}(y-\theta)^2\right)$ for $y \in \bR$.
    \item Exponential densities: $f_{\theta}(y) = \theta \exp(-\theta y) \mathbbm{1}\{y > 0\}$  for $y \in \bR$.
    \item Beta densities with a single shape parameter: $f_{\theta}(y) = \theta y^{\theta-1} \mathbbm{1}\{0<y<1\}$  for $y \in \bR$.
\end{itemize}
For the above examples, the condition that $\mu(\{y \in \bY: T(y) = c\}) = 0$ for $c \in \bR$ also holds, and thus there exists a deterministic optimal rule as shown in Corollary \ref{cor:convexity}.

In addition, Corollary \ref{cor:convexity} applies to the discrete exponential family distributions, with $\mu$ being the counting measure on $\bR$. Examples include Binomial, Geometric, and Poisson distributions.

Finally, if the observation $Y = (Y_1,\ldots,Y_n)$, where $Y_1,\ldots,Y_n$ are independent and identically distributed according to one of the distributions mentioned above, then $Y$ has a density of the form in   \eqref{def:exponential_family} with respect to some product measure. Therefore, Corollary \ref{cor:convexity} applies.


\subsection{Numerical Example - Normal Distributions}

In this subsection, we consider a numerical study using the normal distribution family with a known variance. Specifically, let $\Theta = [-1,1]$ and $\theta_0 = 0$. Let $\mu$ be the Lebesgue measure and $f_{\theta}$ be the density of the normal distribution with mean $\theta$ and variance $1$. By Corollary \ref{cor:convexity} and \ref{cor:exponential_family}, it suffices to consider the following rules:
$$
\tilde{\delta}_{\ell}(y)  = \mathbbm{1}\{y \geq \Phi^{-1}(1 - (\alpha - \Phi(\ell))) \text{ or }  \ y \leq \ell\},
$$
where $-\infty \leq \ell \leq \Phi^{-1}(\alpha)$ and $\Phi$ and $\Phi^{-1}$ are the cumulative distribution function and the quantile function of the standard normal distribution. For $-\infty \leq \ell \leq \Phi^{-1}(\alpha)$, the power function of $\tilde{\delta}_{\ell}$ is
\begin{align*}
p(\theta;\tilde{\delta}_{\ell}) = \int \tilde{\delta}_{\ell}(y) f_{\theta}(y) dy =
\Phi(\ell - \theta) + 1 - \Phi\left(\Phi^{-1}(1 - (\alpha - \Phi(\ell)) - \theta\right).
\end{align*}
Further, the value of the objective function is
$\int_{\Theta_1} g(p(\theta;\tilde{\delta}_{\ell})) \Lambda_1(d\theta)
$. 
A grid search over $\ell \in [-\infty, \Phi^{-1}(\alpha)]$ would yield an optimal rule for the problem in \eqref{def:two_sided_prob}.

Now, consider the following more specific optimization problem:
\begin{align*} 
\begin{split}
& \max_{\delta \in \Delta}\; \beta \int_{-1}^{0} \omega_{v}(p(\theta;\delta)) d\theta
+ \int_{0}^{1} \omega_{v}(p(\theta;\delta)) d\theta
, \\ 
&\text{subject to } p(0;\delta) \leq \alpha,
\end{split}
\end{align*}
where  $\beta > 0$, $v > 0$, and $\omega_v$ is defined in \eqref{def:w_v}. 
The optimal rule is given by $\tilde{\delta}_{\ell^*}$, where $\ell^*$ maximizes the function that maps $\ell \in [-\infty, \Phi^{-1}(\alpha)]$ to
\begin{equation}
\label{def:object_ell}
\beta \int_{-1}^{0} \omega_{v}(p(\theta;\tilde{\delta}_{\ell})) d\theta
+ \int_{0}^{1} \omega_{v}(p(\theta;\tilde{\delta}_{\ell})) d\theta.
\end{equation}

\begin{table*}
\centering
\caption{Optimal $\ell^*$ corresponding to various $\beta$ for the function in \eqref{def:object_ell}.}
\begin{tabular}{|c|c|c|c|c|c|c|c|}
\hline
$\beta$  & 1/3     & 1/2     & 2/3     & 1       & 1.5     & 2       & 3       \\ \hline
$\ell^*$ & -2.3198 & -2.0382 & -1.8587 & -1.6447 & -1.4874 & -1.4101 & -1.3419 \\ \hline
\end{tabular}
\label{tab:prospect}
\end{table*}

We set $\alpha = 10\%$ and $v = 0.69$. For different values of $\beta$, we report the optimal $\ell^*$ in Table \ref{tab:prospect} using numerical integration. Note that for $\beta = 1$, due to symmetry, we have $\ell^* = \Phi^{-1}(\alpha/2)$. Further, in Fig.~\ref{fig:beta_2_over_3}, for $\beta = 2/3$, we plot the objective value as a function of $\ell$ from $-4$ to $\Phi^{-1}(0.1)$.

\begin{figure}[t!]
    \centering
\includegraphics[width=0.75\linewidth]{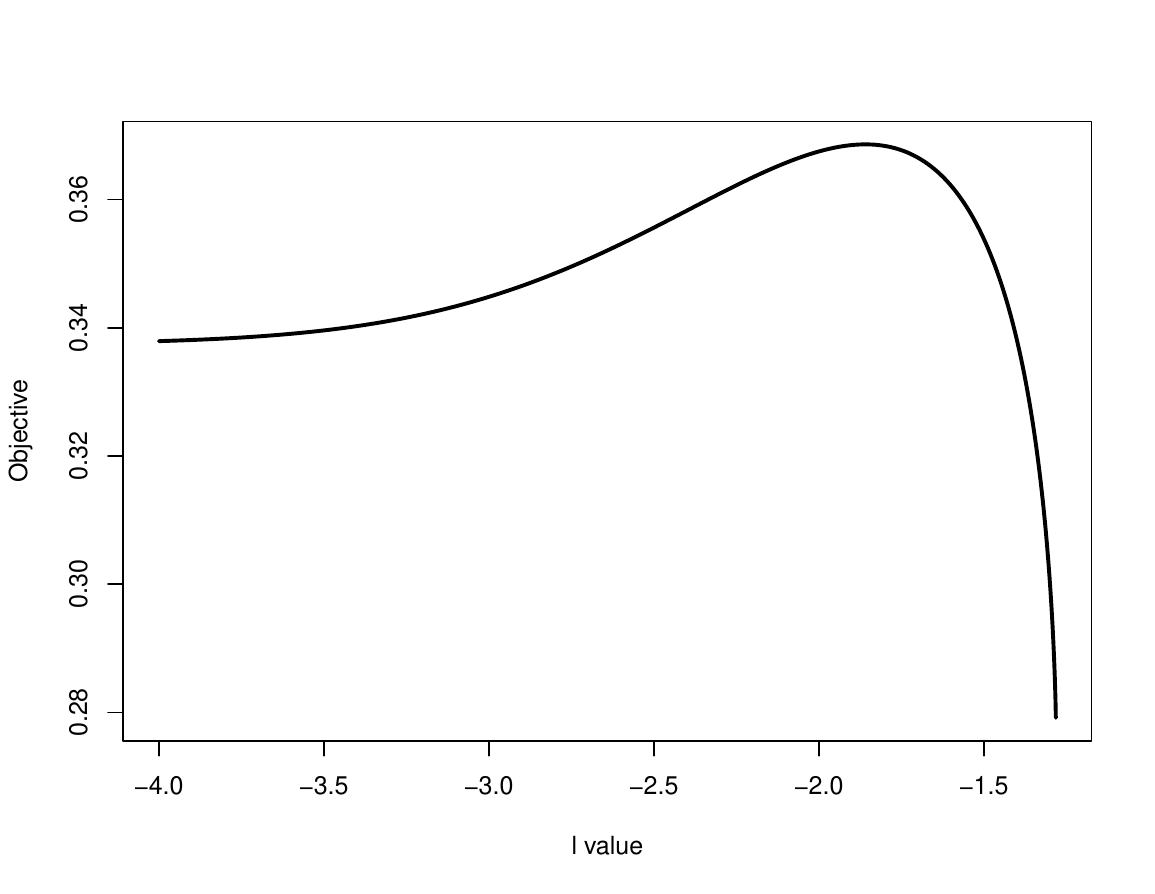}
    \caption{The value of the objective in \eqref{def:object_ell} as a function $\ell$  for $\beta 
    =2/3$ and $v = 0.69$.}
    \label{fig:beta_2_over_3}
\end{figure}

\subsection{Numerical Example - Binomial Distributions}

In this subsection, we consider a numerical study using the Binomial distribution with $n$ trials and success probability $\theta$. Specifically, let $\theta_0 = 0.5$ and $\Theta = [0.4,0.6]$. We aim to solve the following problem: 
\begin{align} \label{def:binomial_prospect}
\begin{split}
& \max_{\delta \in \Delta} \; \beta \int_{0.4}^{0.5} \omega_{v}(p(\theta;\delta)) d\theta
+ \int_{0.5}^{0.6} \omega_{v}(p(\theta;\delta)) d\theta
, \\ 
&\text{subject to } p(0.5;\delta) \leq \alpha,
\end{split}
\end{align}
where we recall that $\omega_{v}(\cdot)$ is defined in \eqref{def:w_v},  $\beta > 0$, and $\alpha \in (0,1)$. By Corollary \ref{cor:convexity} and \ref{cor:exponential_family}, there exists an optimal rule of the following form: for $0\leq y \leq n$,
\begin{align*} 
  \tilde{\delta}_{\ell,u,p_{\ell},p_{u}}(y) = \begin{cases}
      1, \;&\text{ if }  \; y > u \text{ or } y < \ell\\
      p_u, \;&\text{ if }  \;y = u \\
      p_{\ell}, \;&\text{ if }  \;y = \ell \\
      0, \;&\text{ if }  \; \ell < y < u
  \end{cases},
 \end{align*}
 where $0 \leq \ell,u \leq n$ and $0 \leq p_{\ell},p_{u} \leq 1$ are chosen such that
 \begin{align*}
\alpha =     \bP_{\theta_0}(Y > u) +  \bP_{\theta_0}(Y < \ell) +      p_{\ell} \times \bP_{\theta_0}(Y =  \ell) +
          p_{u} \times \bP_{\theta_0}(Y = u).
 \end{align*} 
 Here, $\bP_{\theta_0}$ means that $Y$ has the binomial distribution with $n$ trials and success probability given by $\theta_0$. Once $\ell$ and $p_{\ell}$ are chosen, $u$ and $p_{u}$ can be solved based on the above equation if they exist. Thus a grid search over $\ell$ and $p_{\ell}$ would yield an optimal solution.

For $n = 10$, $\alpha = 9\%$, and $v = 0.69$, we report in Table \ref{tab:binomial_prospect} the optimal values of $(\ell^*,u^*,p_{\ell^*},p_{u^*})$ for various $\beta$.
Note that due to symmetry, for $\beta = 1$, we have $\ell^* + u^* = 10$ and $p_{\ell^*} = p_{u^*}$. Further, in Fig.~\ref{fig:binomial}, we plot the objective function value as a function of $p_{\ell}$ for $\beta = 1.05$ and $\ell = 2$.

\begin{table}[t!]
\centering
\caption{Optimal $\ell^*,u^*,p_{\ell^*},p_{u^*}$ corresponding to  various $\beta$ for Problem \eqref{def:binomial_prospect}.}
\begin{tabular}{|c|c|c|c|c|c|}
\hline
$\beta$  & 0.5    & 0.97     & 1    & 1.05      & 1.2          \\ \hline
$\ell^*$ & 1 & 2 & 2 & 2 & 2    \\ \hline
$u^*$ & 7 & 8 & 8 & 8 & 8    \\ \hline
$p_{\ell}^*$ & 1 & 0.7189 & 0.7799 & 0.8776 & 1    \\ \hline
$p_{u}^*$ & 0.2097 & 0.8402 & 0.7792 & 0.6814 & 0.5591    \\ \hline
\end{tabular}
\label{tab:binomial_prospect}
\end{table}

\begin{figure}[t!]
    \centering
\includegraphics[width=0.75\linewidth]{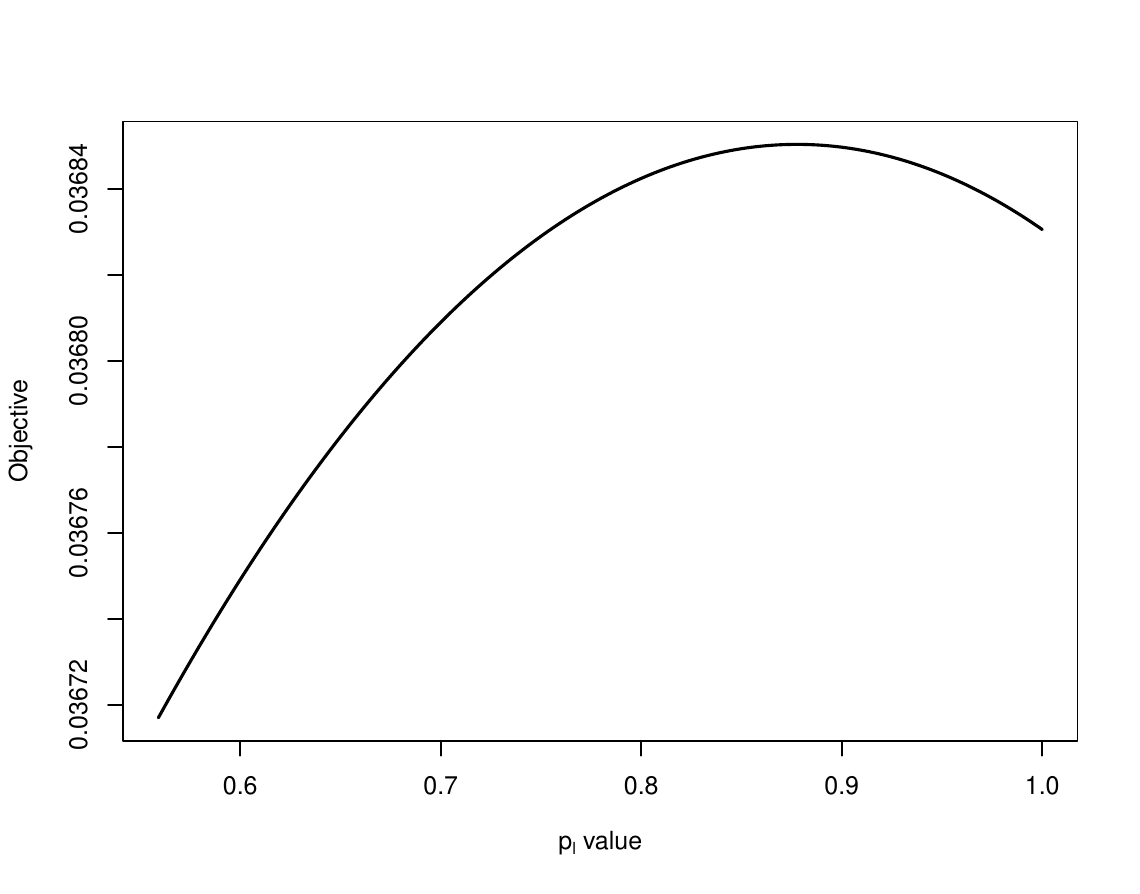}
    \caption{The value of the objective in \eqref{def:binomial_prospect} as a function $p_{\ell}$  for $\ell = 2$, $\beta 
    =1.05$, and $v = 0.69$.}
    \label{fig:binomial}
\end{figure}

\subsection{Discussions}
As previously discussed, we focus on cases where no UMP test exists. However, Theorem \ref{theorem:simple_null_structrual} also applies when UMP tests do exist, recovering their existence and structural properties, which is well known in the literature \cite{romano2005testing}.

Specifically, assume that $g(t) =t$ for $t \in [0,1]$ in this subsection. When $\Theta_1$ contains a single element, say $\theta_1$, the probability measure $\Lambda_1$ is the Dirac measure at $\theta_1$. The problem \eqref{def:two_sided_prob} reduces to the following:
\begin{align*}
    \max_{\delta \in \Delta}\ p(\theta_1;\delta), \;\; \text{ subject to } p(\theta_0;\delta) \leq \alpha.
\end{align*}
Then, Theorem \ref{theorem:simple_null_structrual} reduces to the well-known NP lemma \cite[Theorem 3.2.1]{romano2005testing},
since $H(\cdot)$ in \eqref{def:h_function} is $f_{\theta_1}(\cdot)$.

Further, consider the case that the parameter is a scalar, i.e., $\Theta \subset \bR$. Assume the following monotone likelihood ratio (MLR) property \cite{romano2005testing} holds: there exists a measurable function $T:\bY \to \bR$ such that $f_{\theta'}(y)/f_{\theta}(y)$ is non-decreasing function of $T(y)$ for any $\theta, \theta' \in \Theta$ and $\theta < \theta'$. If $\theta_0 < \theta$ for $\theta \in \Theta_1$, then $H(y)$ is a non-decreasing function of $T(y)$. By  Theorem \ref{theorem:simple_null_structrual}, for any probability measure $\Lambda_1$ on $\Theta_1$, there exists an optimal rule for problem \eqref{def:two_sided_prob} of the following form: for $y \in \bY$,
\begin{align*} 
  \delta^*(y) = \begin{cases}
      1, \;&\text{ if }  \;T(y) > \kappa \\
      \gamma(y), \;&\text{ if }  \;T(y) = \kappa \\
      0, \;&\text{ if }  \;T(y) < \kappa
  \end{cases}.
 \end{align*}
In general, $\gamma(\cdot)$ depends on $\Lambda_1$, so Theorem \ref{theorem:simple_null_structrual} does not immediately imply the existence of UMP tests. However,
if we assume that for any probability measure $\nu$ on $\Theta_1$ and $c \geq 0$, 
\begin{align*}
\mu\left(\left\{y \in \bY: \int_{\Theta_1} {f_{\theta}(y)} \nu(d\theta) = c f_{\theta_0}(y) \right\}\right) = 0,
\end{align*}   
then Theorem \ref{theorem:simple_null_structrual} establishes that there exists a UMP test of the form: $\delta^*(y) = \mathbbm{1}\{T(y) \geq \kappa\}$ for $y \in \bY$, where $\kappa$ is selected so that $p(\theta_0;\delta^*) = \alpha$.

\section{Composite Null and  Alternative: Integrated False-Alarm  Control}\label{sec:composite_null_integrated}

In this section, we consider the case where both $\Theta_0$ and $\Theta_1$ are composite, and aim to test:
\begin{align}
  \label{def:composite_null}
{\mathcal{H}}_0: \theta \in \Theta_0, \text{ vs }\, {\mathcal{H}}_1: \theta \in \Theta_1 := \Theta \setminus \Theta_0.
\end{align}
Let $\Lambda_0$ and $\Lambda_1$ be two probability measures on $\Theta_0$ and $\Theta_1$ respectively, $\alpha \in (0,1)$  a user-specified level, and $g:[0,1] \to \bR$  a measurable function. Our goal is to solve the following problem:
\begin{align} \label{def:two_sided_prob_composite}
\begin{split}
    &\max_{\delta \in \Delta} \int_{\Theta_1} g(p(\theta;\delta)) \Lambda_1(d\theta), \\ 
&\text{subject to } \int_{\Theta_0} p(\theta;\delta) \Lambda_0(d\theta) \leq \alpha.
\end{split}
\end{align}
This generalizes the case in \eqref{def:two_sided_prob}. Define the following probability density function w.r.t.~$\mu$: for $y \in \bY$,
\begin{align*}
    \tilde{f}_{\Theta_0}(y) := \int_{\Theta_0} f_{\theta}(y) \Lambda_0(d\theta).
\end{align*}
We have the following structural result, which parallels Theorem \ref{theorem:simple_null_structrual}, with $f_{\theta_0}$ replaced by $\tilde{f}_{\Theta_0}$.


\begin{theorem}  
\label{thrm:composite_null_avg}
    Suppose that Assumptions \ref{assumption:continuity} and \ref{assumption_g_differentiable} hold, and that 
$\bar{\delta}_0$ is not an optimal solution to problem \eqref{def:two_sided_prob_composite}.
\begin{enumerate}[label=\roman*)]
    \item There exist  a rule $\delta^* \in \Delta$, constant $\kappa \in \bR$, and a measurable function $\gamma: \bY \to [0,1]$ such that  $\delta^*$ solves the optimization problem in \eqref{def:two_sided_prob_composite}, and for each $y \in \bY$,
 \begin{align*} 
  \delta^*(y) = \begin{cases}
      1, \;&\text{ if }  \;{H}(y) > \kappa \tilde{f}_{\Theta_0}(y) \\
      \gamma(y), \;&\text{ if }  \;{H}(y) = \kappa \tilde{f}_{\Theta_0}(y) \\
      0, \;&\text{ if }  \;{H}(y) < \kappa \tilde{f}_{\Theta_0}(y)
  \end{cases},
 \end{align*}
where recall that $H(\cdot)$ is defined in \eqref{def:h_function}.


\item For any optimal rule $\delta^* \in \Delta$, it must take the above form  for $\mu$-a.e.~$y$.

\item If, in addition to Assumptions \ref{assumption:continuity} and \ref{assumption_g_differentiable}, Assumption \ref{assumption_g_increasing} holds, then $\bar{\delta}_0$ is not an optimal solution to problem \eqref{def:two_sided_prob_composite}, and for any optimal rule $\delta^*\in \Delta$, we have $\int_{\Theta_0} p(\theta; \delta^*) \Lambda_0(d\theta) = \alpha$.
\end{enumerate}
\end{theorem}
\begin{proof}
    See Appendix \ref{app:proof_composite_null}.
\end{proof}

We now simplify the above results for the single-parameter exponential family distributions defined in \eqref{def:exponential_family} for ``two-sided'' hypotheses with a composite null.

\begin{theorem}\label{theorem:composiste_null_exponential}
Consider the single-parameter exponential family distributions in \eqref{def:exponential_family}. Suppose that Assumptions   \ref{assumption_g_differentiable} and \ref{assumption_g_increasing} hold, and that $\Theta_0 = [a,b]$ and $\Theta = [-K,K]$ for some $K > \max\{|a|, |b|\}$.  There exists an optimal rule $\delta^*$ for the problem in \eqref{def:two_sided_prob_composite}  such that $\int_{a}^{b} p(\theta; \delta^*) \Lambda_0(d\theta) = \alpha$, and for some constants $-\infty \leq \ell < u \leq \infty$ and a measurable function $\gamma:\bY \to [0,1]$, $\delta^*$ is given by \eqref{def:optimal_rule_Ty}.
\end{theorem}
\begin{proof}
    See Appendix \ref{app:proof_composite_null}.
\end{proof}

\begin{remark}\label{remark:composite_avg}
    The proof strategy of Theorem \ref{theorem:composiste_null_exponential} is as follows. Note that $\log({H}(y)/ \tilde{f}_{\Theta_0}(y))$ is a function of $T(y)$,   denoted by $\psi(T(y))$. Using a change-of-measure argument (see Lemma \ref{lemma:change_of_measure}), we show that if $\psi'(x) = 0$  for some $x \in \bR$, then $\psi''(x)$ must be positive. By Lemma \ref{lemma:roots}, this implies that $\psi'$ has \emph{at most} one root. Thus,  $\psi$ is either monotonic or has a single inflection point.
\end{remark}

\subsection{Numerical Example}

We use a numerical example to illustrate the results in Theorem \ref{theorem:composiste_null_exponential}.   Let $\mu$ be the Lebesgue measure and $f_{\theta}$ be the density of the normal distribution with mean $\theta$ and variance $1$. Let $\Theta_0 = [-0.2,0.4]$ and $\Theta = [-1,1]$. Consider the following optimization problem:
\begin{align*}
    \max_{\delta \in \Delta} \int_{-1}^{-0.2} \sqrt{p(\theta;\delta)} d\theta + \int_{0.4}^{1} \sqrt{p(\theta;\delta)} d\theta,
\end{align*}
subject to the constraint that  $\int_{-0.2}^{0.4} p(\theta;\delta) d\theta \leq 0.1$. By Theorem \ref{theorem:composiste_null_exponential}, there exists an optimal rule of the following form: for each $y \in \bR$,
\begin{align*}
    \tilde{\delta}_{\ell}(y) = \mathbbm{1}\left\{y \leq \ell \; \text{ or }\; y \geq u_{\ell} \right\},
\end{align*}
where $\ell$ and $u_{\ell}$ satisfy the following constraint 
$\int_{-0.2}^{0.4} {p(\theta;\tilde{\delta}_{\ell})} d\theta = 0.1$. Then, a grid search over $\ell$ would yield an optimal solution. 
By numerical integration, the optimal parameters are obtained as $\ell^* = -1.1673$ and $u_{\ell^*}= 1.6713$.

\section{Composite Null and Alternative: supremum False-Alarm Control}\label{sec:composite_null_supremum}

In this section, we continue considering the hypotheses in \eqref{def:composite_null}. However, rather than controlling the integrated false-alarm risk, we focus on the supremum risk. Specifically,  the objective is to solve the following optimization problem:
\begin{align} \label{def:composite_null_sup}
\begin{split}
    &\max_{\delta \in \Delta} \int_{\Theta_1} g(p(\theta;\delta)) \Lambda_1(d\theta), \\ 
&\text{subject to } \sup_{\theta \in\Theta_0} p(\theta;\delta)   \leq \alpha,
\end{split}
\end{align}
where  $\Lambda_1$ is a probability measure on $\Theta_1$, $\alpha \in (0,1)$  a user-specified level, and $g:[0,1] \to \bR$  a measurable function. In this setup, a UMP test typically does not exist.

\begin{theorem}\label{theorem:composiste_supremum}
  Suppose that Assumptions \ref{assumption:continuity}, \ref{assumption_g_differentiable}, and \ref{assumption_g_increasing} hold, and that $\Theta_0$ is closed. There exists an optimal rule $\delta^* \in \Delta$ for problem \eqref{def:composite_null_sup} such that 
  $\sup_{\theta \in \Theta_0} p(\theta; \delta^*) = \alpha$
  and that for  some constant $\kappa \geq 0$, a measurable function $\gamma:\bY \to [0,1]$ and a probability measure $\widetilde{\Lambda}_0$ on $\Theta_0$, we have
 \begin{align*}
  \delta^*(y) = \begin{cases}
      1, \;&\text{ if }  \;H(y) > \kappa {\widetilde{f}_{\widetilde{\Lambda}_0}(y)} \\
      \gamma(y), \;&\text{ if }  \;H(y) = \kappa 
 {\widetilde{f}_{\widetilde{\Lambda}_0}(y)}\\
      0, \;&\text{ if }  \;H(y) < \kappa {\widetilde{f}_{\widetilde{\Lambda}_0}(y)}
  \end{cases},
 \end{align*}
  for each $y \in \bY$, 
 where $H(\cdot)$ is defined in \eqref{def:h_function} and 
 $$
 \widetilde{f}_{\widetilde{\Lambda}_0}(y) := \int f_{\theta}(y) \widetilde{\Lambda}_0(d\theta).
 $$
\end{theorem}

\begin{proof}
    See Appendix \ref{app:composite_supremum}.
\end{proof}

\begin{remark}
From the proof of Theorem \ref{theorem:composiste_supremum}, $\widetilde{\Lambda}_0$ is the least-favorable distribution (see \cite[Theorem 3.8.1]{romano2005testing})  for the following hypothesis testing problem under the supremum false-alarm constraint: the null hypothesis $\{f_{\theta}: \theta \in \Theta_0\}$ versus the simple alternative (normalized) $H(\cdot)$.
\end{remark}

We now focus on two-sided hypotheses and exponential family distributions to simplify the preceding results.

\begin{theorem}\label{theorem:composiste_exponential_supremum}
Consider the single-parameter exponential family distributions in \eqref{def:exponential_family}, and assume that $\Theta_0 = [a,b]$ and $\Theta = [-K,K]$ for some $K > \max\{|a|,|b|\}$. Suppose that Assumptions \ref{assumption_g_differentiable} and \ref{assumption_g_increasing} hold, and that $T(y)$ takes more than two values, that is, $\mu(\{y \in \bY: T(y) \not\in \{c_1,c_2\}\}) > 0$ for any $c_1,c_2 \in \bR$.     Then, there exists an optimal rule $\delta^*$ for  problem  \eqref{def:composite_null_sup}  such that $\sup_{\theta \in [a,b]} p(\theta; \delta^*)  = \alpha = \max\{p(a;\delta), p(b;\delta)\}$, and that for some  $-\infty \leq \ell < u \leq \infty$ and measurable function $\gamma:\bY \to [0,1]$, $\delta^*$ is given by \eqref{def:optimal_rule_Ty}.
\end{theorem}
\begin{proof}
    See Appendix \ref{app:composite_supremum}.
\end{proof}
 
\begin{remark}
In addition to the strategy discussed in Remark \ref{remark:composite_avg}, the key to the proof of Theorem \ref{theorem:composiste_exponential_supremum} is to show that, for a rule of the form \eqref{def:optimal_rule_Ty}, the worst-case false-alarm probability over 
$[a,b]$ is attained at either $a$ or $b$.
\end{remark}
\subsection{Numerical Example}

Let $\mu$ be the Lebesgue measure and $f_{\theta}$ be the density of the normal distribution with mean $\theta$ and variance $1$. Let $a = -0.2, b = 0.2$ and $K=1$. Consider the following optimization problem:
\begin{align*}
    \max_{\delta \in \Delta} \beta \int_{-1}^{-0.2} \sqrt{p(\theta;\delta)} d\theta + \int_{0.2}^{1} \sqrt{p(\theta;\delta)} d\theta,
\end{align*}
subject to the constraint that  $\sup_{\theta \in [-0.2,0.2]} p(\theta;\delta)  \leq \alpha$. By Theorem  \ref{theorem:composiste_exponential_supremum}, there exists an optimal rule of the following form: for each $y \in \bR$,
$\tilde{\delta}_{\ell}(y) = \mathbbm{1}\left\{y \leq \ell \; \text{ or }\; y \geq u_{\ell} \right\}$,
where $\ell$ and $u_{\ell}$ satisfy the following constraint 
$\max\{p(-0.2;\tilde{\delta}_{\ell}), p(0.2;\tilde{\delta}_{\ell})\}= \alpha$. Then, a grid search over $\ell$ would yield an optimal solution. 
By numerical integration, for $\alpha = 10\%$ and various $\beta$, we report in Table \ref{tab:comp_sup} the optimal pairs of $(\ell^*,u_{\ell^*})$, as well as the corresponding rejection probabilities at $-0.2$ and $0.2$.

\begin{table}[t!]
\centering
\caption{Optimal pairs of $(\ell^*,u_{\ell^*})$ and the rejection probabilities for the supremum false-alarm control example}
\begin{tabular}{|c|c|c|c|c|}
\hline
$\beta$ & $\ell^*$ & $u^*$ & $p(a;\tilde{\delta}_{\ell^*})$ & $p(b; \tilde{\delta}_{u_{\ell^*}})$ \\ \hline
1                    & -1.677   & 1.677 & $10\%$                   & $10\%$                        \\ \hline
5                    & -1.559   & 2.027 & $10\%$                    & $7.31\%$                     \\ \hline
10                   & -1.505   & 2.444 & $10\%$                    & $5.65\%$                     \\ \hline
\end{tabular}
\label{tab:comp_sup}
\end{table}

\section{Conclusion}

In this paper, we investigated the problem of composite binary hypothesis testing within the NP framework. We demonstrated that for any hypothesis testing problem involving rejection probabilities, it is sufficient to consider the family of ``generalized Bayes rules''. By analyzing several special cases, we provided more concrete results for optimal decision rules. Unlike existing studies, we allow the objective function to incorporate a nonlinear component in the detection probabilities. As a result, our findings are applicable to a wide range of detection problems, including those involving behavioral utility-based hypothesis testing.

\appendices

\section{Proofs of Lemma \ref{lemma:structural} and Lemma \ref{lemma:Bayes_rule}}\label{app:proof_cP}

\begin{proof}[Proof of Lemma \ref{lemma:structural}]
Let $\theta_n, n \geq 1$ and $\theta$ be in $\Theta$ such that $\lim_{n\to \infty} \theta_n = \theta$. For any procedure $\delta \in \Delta$, due to Assumption \ref{assumption:continuity}, we have
\begin{align*}
p(\theta_n; \delta) = &\int_{\bY} \delta(y) f_{\theta_n}(y) \mu(dy)\\
\to &\int_{\bY} \delta(y) f_{\theta}(y) \mu(dy) = p(\theta; \delta),
\end{align*}
and thus $p(\cdot;\delta) \in C(\Theta)$, which completes the proof of the first statement.

Further, let $p_1,p_2 \in \cP$. By definition, there exist $\delta_1,\delta_2 \in \Delta$ such that $p_k(\cdot) = p(\cdot;\delta_k)$ for $k = 1,2$. For any $\alpha \in (0,1)$,
\begin{align*}
    \alpha p_1(\cdot) + (1-\alpha) p_2(\cdot)
    &=     \alpha p(\cdot;\delta_1) + (1-\alpha) p(\cdot;\delta_2)\\
    &= p(\cdot; \alpha \delta_1 + (1-\alpha)\delta_2).
\end{align*}
That is, $\alpha p_1 + (1-\alpha) p_2$ is the power function of $\alpha \delta_1 + (1-\alpha) \delta_2$, and belongs to $\cP$. Thus, $\mathcal{P}$ is convex.

Next, we show that $\mathcal{P}$ is compact. Note that
\begin{align*}
 \sup_{\delta \in \Delta}  \left| p(\theta_n;\delta) - p(\theta;\delta)\right| \leq \int_{\bY} |f_{\theta_n}(y) - f_{\theta}(y)| \mu(dy).
\end{align*}
Thus, due to Assumption \ref{assumption:continuity}, $\cP \subset C(\Theta)$ is equicontinuous at any $\theta \in \Theta$. Further, by definition, $\sup_{\theta \in \Theta, p \in \cP} |p(\theta)| \leq 1$. Thus, by Arzela-Ascoli Theorem  \cite[Theorem 4.43]{folland1999real}, the closure of $\cP$ is compact in $C(\Theta)$. It remains to show that $\cP$ is closed.

Let $p_n, n \geq 1$ be elements in $\cP$ such that $\|p_n - p\|_{\infty} \to 0$ as $n \to \infty$ for some $p \in C(\Theta)$. By definition, there exist $\delta_n, n \geq 1$ in $\Delta$ such that
$p_n(\cdot) = p(\cdot; \delta_n)$ for $n \geq 1$. By \cite[Theorem A.5.1]{romano2005testing}, there exists a subsequence $\delta_{n_j}, j \geq 1$  and $\delta \in \Delta$, such that
$$
\lim_{j \to \infty} \int_{\bY} \delta_{n_j}(y) \eta(y) \mu(dy) = \int_{\bY}\delta(y) \eta(y) \mu(dy),
$$
for any $\mu$-integrable function $\eta$. As a result, for any $\theta \in \Theta$,
\begin{align*}
    p(\theta) &= 
\lim_{j \to \infty} \int_{\bY} \delta_{n_j}(y) f_{\theta}(y) \mu(dy)\\ 
&= \int_{\bY}\delta(y) f_{\theta}(y) \mu(dy) = p(\theta;\delta),
\end{align*}
which implies that $p \in \cP$. Thus $\cP$ is closed, and the proof of the second statement is complete.

Next, we consider the third statement. For any $\delta \in \Delta$, by the dominated convergence theorem \cite[Theorem 2.27]{folland1999real}, the function $p(\cdot;\delta) \in C(\Theta)$ is differentiable on $U$. Let $\tilde{f} \in C(\Theta)$ be any non-differentiable function on $U$. Then 
$p(\cdot;\delta) + \epsilon \tilde{f}$ is not differentiable on $U$ for any $\epsilon > 0$ and thus does not belong to $\cP$. This proves that $\cP$ has no interior point.

Finally, we consider the last statement. Since $p \in \cP$, it is the power function of some decision rule $\delta \in \Delta$, that is, 
$p(\theta) = \int \delta(y) f_{\theta}(y) \mu(dy)$ for each $\theta \in \Theta$.  
Recall that $\mu_{\theta}$ is the distribution of $Y$ when $\theta$ is the true parameter. Since $p(\theta') \in (0,1) $, we have
\begin{align*}
&\mu_{\theta'}(\{y \in \bY: \delta(y) > 0\}) > 0, \\
& \mu_{\theta'}(\{y \in \bY: \delta(y) < 1\}) > 0.
\end{align*}
Due to Assumption \ref{assumption:continuity},  we have that $\mu_{\theta}(\{y \in \bY: \delta(y) > 0\}) > 0$
and $\mu_{\theta}(\{y \in \bY: \delta(y) < 1\}) > 0$ 
for all $\theta \in \Theta$, which implies that $0< p(\theta) < 1$ for all $\theta \in \Theta$. Then the proof is complete since $p \in C(\Theta)$ and $\Theta$ is compact.
\end{proof}

\begin{proof}[Proof of Lemma \ref{lemma:Bayes_rule}]
    By assumption, $\nu = \nu^+ - \nu^{-}$, and $\nu^+$ and $\nu^-$ are positive, finite measures, and mutually singular. 
 Note that by the Fubini's Theorem,
 \begin{align*}
\int_{\Theta} p(\theta; \delta) \nu(d\theta) &= 
\int_{\Theta} \left( \int_\mathbbm{Y} \delta(y) f_{\theta}(y) \mu(dy)\right) 
\nu(d\theta) \\
&=\int_\mathbbm{Y} \delta(y) 
\left(\int_{\Theta} f_{\theta}(y) \nu^+(d\theta) 
- \int_{\Theta} f_{\theta}(y) \nu^-(d\theta)
\right)
\mu(dy).
 \end{align*}
 Then the result is immediate.
\end{proof}

\section{Proof of Theorem \ref{thrm:general_case}}\label{app:proof_general_case}

Recall that $\Theta$ is compact and $C(\Theta)$ is the space of continuous functions on $\Theta$ equipped the supremum norm $\|\cdot\|_{\infty}$, and that $\cM(\Theta)$ is the space of finite signed measures. For each $\nu \in \cM(\Theta)$, denote by $|\nu|$ the total variation of $\nu$. Note that the dual space (i.e. the space of bounded linear functionals) of $C(\Theta)$ is   $\cM(\Theta)$ \cite[Theorem 7.17]{folland1999real}.

We first define the support points of a subset   $\mathcal{F} \subset C(\Theta)$, and show that if $\mathcal{F}$ is non-empty, convex and closed with an empty interior, then all points of $\mathcal{F}$ are support points.

\begin{definition}
Let $\mathcal{F} \subset C(\Theta)$. We say $f \in \mathcal{F}$ is a support point of $\mathcal{F}$ if there exists a \emph{non-zero} finite signed measure $\nu \in \mathcal{M}(\Theta)$ such that
$$
\int_{\Theta} f(\theta) \nu(d\theta) \geq \int_{\Theta}  h(\theta) \nu(d\theta),\;\; \text{ for } h \in \mathcal{F}.
$$
\end{definition}

\begin{lemma}\label{lemma:C_support}
Let $\mathcal{F} \subset C(\Theta)$ be non-empty, closed and convex. Assume $\mathcal{F}$ has an empty interior. Then any point $f \in \mathcal{F}$ is a support point of $\mathcal{F}$.
\end{lemma}
\begin{proof}
Let $f \in \mathcal{F}$ be arbitrary. 
Since $\mathcal{F}$ is convex and closed and has an empty interior, 
by Bishop–Phelps Theorem \cite[Theorem 7.43]{aliprantis2006infinite}, the set of {support points} of $\mathcal{F}\subset C(\Theta)$ is \emph{dense} in $\mathcal{F}$. Thus, there exists a sequence $\{f_n: n \geq 1\} \subset \mathcal{F}$ such that $f_n$ is a support point of $\mathcal{F}$ for $n \geq 1$, and $\|f_n - f\|_{\infty} \to 0$ as $n \to \infty$.

By definition, for each $n \geq 1$, there exists a non-zero finite signed measure $\nu_n \in \mathcal{M}(\Theta)$ such that
\begin{equation}
    \label{aux:support_n}
\int_{\Theta} f_n(\theta) \nu_n(d\theta) \geq \int_{\Theta}  h(\theta) \nu_n(d\theta), \text{ for } h \in \mathcal{F}.
\end{equation}
Without loss of generality, we may assume $|\nu_n| = 1$. Since $\Theta$ is compact, by applying the Prohorov's Theorem \cite[Theorem 13.29]{klenke2013probability} to the positive and negative parts of $\nu_n, n\geq 1$, there exists $\nu_{\infty} \in \mathcal{M}(\Theta)$ and a subsequence $\{\nu_{n_k}: k \geq 1\}$ of $\{\nu_n: n \geq 1\}$ such that $\nu_{n_k}$ converges to $\nu_{\infty}$ weakly, that is, for any $h \in C(\Theta)$, we have for each $h \in C(\Theta)$,
$$
\lim_{k \to \infty} \int_{\Theta} h(\theta) \nu_{n_k}(d\theta) 
= \int_{\Theta} h(\theta) \nu_{\infty}(d\theta).
$$
Further, by Portmanteau theorem \cite[Theorem 13.16]{klenke2013probability}, we have $|\nu_{\infty}| = 1$, that is, $\nu_{\infty}$ is non-zero.

Since $\|f_n - f\|_{\infty} \to 0$, $|\nu_{n_k}| = 1$ for $k\geq 1$,  and $\nu_{n_k}$ converges to $\nu_{\infty}$ weakly, we have
\begin{align*}
 &   \lim_{k \to \infty} \int_{\Theta} f_{n_k}(\theta) \nu_{n_k}(d\theta)  \\
=&  \lim_{k \to \infty} \int_{\Theta} (f_{n_k}(\theta)  - f(\theta)) \nu_{n_k}(d\theta)   
  +   \int_{\Theta}   f(\theta) \nu_{n_k}(d\theta) \\
= &\int_{\Theta}   f(\theta) \nu_{\infty}(d\theta). 
\end{align*}
Further,  for any $h \in \mathcal{F}$, due to \eqref{aux:support_n},
\begin{align*}
\int_{\Theta}  h(\theta) \nu_{\infty}(d\theta) 
=&\lim_{k \to \infty} \int_{\Theta}  h(\theta) \nu_{n_k}(d\theta) \\
\leq &\lim_{k \to \infty} \int_{\Theta}  f_{n_k}(\theta) \nu_{n_k}(d\theta).
\end{align*}
Thus, for any $h \in \mathcal{F}$, we have $
\int_{\Theta}   f(\theta) \nu_{\infty}(d\theta)  \geq \int_{\Theta}  h(\theta) \nu_{\infty}(d\theta)$, which shows that $f$ is a support point of $\mathcal{F}$. The proof is complete.
\end{proof}

\begin{proof}[Proof of Theorem \ref{thrm:general_case}]
Let $p \in \cP$. By definition, $p$ is the power function of some decision rule $\delta^* \in \Delta$, that is, $p(\cdot) = p(\cdot; \delta^*)$.

By Lemma \ref{lemma:structural} and the assumption, $\cP$ is non-empty, compact and convex, with an empty interior. Then, by Lemma \ref{lemma:C_support}, $p$ is a support point of $\cP$, that is, there exists some non-zero finite signed measure $\nu \in \mathcal{M}(\Theta)$ such that
$$
\int_{\Theta} p(\theta) \nu(d\theta) \geq \int_{\Theta}  h(\theta) \nu(d\theta),\;\; \text{ for } h \in \mathcal{P}.
$$

Since $\cP$ is the set of all power functions, it implies that $\delta^*$ is an optimal solution to the following problem:
$$
\max_{\delta \in \Delta}\; \int_{\Theta} p(\theta;\delta) \nu(d\theta).
$$

Finally, by Lemma \ref{lemma:Bayes_rule} and the remark following it, $\delta^*$ must be of the form in \eqref{def:optimal_rule}. The proof is complete.
\end{proof}

\section{Proof of Theorem  \ref{theorem:simple_null_structrual}}\label{app:proof_two_sided_theorem}
 
\begin{proof}[Proof of Theorem \ref{theorem:simple_null_structrual}]
By Lemma \ref{lemma:structural}, $\cP$ is compact, which implies that $\{p \in \cP: p(\theta_0) \leq \alpha\}$ is also compact.  Due to Assumption \ref{assumption_g_differentiable}, the function $p \in \cP \mapsto \int_{\Theta_1} g(p(\theta)) \Lambda_1(d\theta) \in \bR$ is continuous. Thus, there exists $p^* \in \mathcal{P}$ that solves the following  problem:
\begin{align*}
\max_{p \in \mathcal{P}} \int_{\Theta_1} g(p(\theta)) \Lambda_1(d\theta), \quad
\text{subject to } p(\theta_0) \leq \alpha.
\end{align*}
Since $\bar{\delta}_0$ is not an optimal solution to problem \eqref{def:two_sided_prob}, we have  $\Lambda_1(\{\theta \in \Theta_1: p^*(\theta) > 0\}) > 0$. Further, since $p^*$ is feasible, we have  
$p^*(\theta_0) \leq \alpha$. By Lemma \ref{lemma:structural}, 
\begin{equation}\label{aux:bounded_away_01}
    0 < \inf_{\theta \in \Theta} p^*(\theta) \leq \sup_{\theta \in \Theta} p^*(\theta) < 1.
\end{equation}
Since $p^* \in \cP$, there exists $\delta^* \in \Delta$ such that $p^*(\cdot) = p(\cdot;\delta^*)$.

Fix an arbitrary rule $\tilde{\delta} \in \Delta$ such that its associated power function $\tilde{p}$ satisfies $\tilde{p}(\theta_0) \leq \alpha$, that is, $\tilde{\delta}$ is a feasible solution for the  problem in \eqref{def:two_sided_prob}. Since $\cP$ is convex by Lemma \ref{lemma:structural}, $(1-\epsilon) p^* + \epsilon \tilde{p} \in \cP$ and $(1-\epsilon) p^*(\theta_0) + \epsilon \tilde{p}(\theta_0)\leq \alpha$ for $\epsilon \in [0,1]$. Define the following function: for $\epsilon \in [0,1]$,
$$
J(\epsilon) = \int_{\Theta_1} g\left((1-\epsilon) p^*(\theta) + \epsilon \tilde{p}(\theta)\right) \Lambda_1(d\theta).
$$
Since $g(\cdot)$ is continuously differentiable on $(0,1)$ (see Assumption \ref{assumption_g_differentiable}), due to \eqref{aux:bounded_away_01} and by the dominated convergence theorem, $J'(\epsilon)$ is equal to the following
$$
\int_{\Theta_1} g'\left((1-\epsilon) p^*(\theta) + \epsilon \tilde{p}(\theta)\right)(\tilde{p}(\theta) - p^*(\theta)) \Lambda_1(d\theta).
$$

By the optimality of $p^*$, $J(\cdot)$ attains its maximal value at $\epsilon = 0$, which implies $J'(0) \leq 0$, that is,
$$
\int_{\Theta_1} g'\left(p^*(\theta)\right)\tilde{p}(\theta)  \Lambda_1(d\theta)
\leq \int_{\Theta_1} g'\left(p^*(\theta)\right) p^*(\theta)\Lambda_1(d\theta).
$$
Recall the definition of $H(\cdot)$ in \eqref{def:h_function}. By definition and Fubini's theorem,
\begin{align*}
    \int_{\Theta_1} g'\left(p^*(\theta)\right)\tilde{p}(\theta)  \Lambda_1(d\theta)       
&=   \int_{\Theta_1} g'\left(p^*(\theta)\right) \left(\int_{\bY} \tilde{\delta}(y) f_{\theta}(y) \mu(dy)\right) \Lambda_1(d\theta)  \\
&=\int_{\bY} \tilde{\delta}(y)  H(y) \mu(dy).
\end{align*}
Similarly, $$
\int_{\Theta_1} g'\left(p^*(\theta)\right){p}^*(\theta)  \Lambda_1(d\theta) =   \int_{\bY} {\delta^*}(y) H(y) \mu(dy).
$$

Since $\tilde{\delta} \in \Delta$ is arbitrary, we must have that $\delta^*$ solves the following optimization problem:
\begin{align*}
    \begin{split}
&\max_{\delta \in \Delta}   \int_{\bY} {\delta}(y) H(y) \mu(dy), \\
&\text{subject to } \int_{\bY} {\delta}(y) {f}_{\theta_0}(y) \mu(dy) \leq \alpha.
\end{split}
\end{align*}
Define $c^*:= \int_{\bY} \delta^*(y) f_{\theta_0}(y) \mu(dy)$.  Then, 
$\delta^*$ also solves the following optimization problem:
\begin{align*}
    \begin{split}
&\max_{\delta \in \Delta}   \int_{\bY} {\delta}(y) H(y) \mu(dy), \\
&\text{subject to } \int_{\bY} {\delta}(y) {f}_{\theta_0}(y) \mu(dy) = c^*.
\end{split}
\end{align*}

Since $\delta^*$ is feasible, and due to \eqref{aux:bounded_away_01}, $c^* \in (0,\alpha] \subset (0,1)$. Further, since  $g'$ is continuous on $(0,1)$ (see Assumption \ref{assumption_g_differentiable}),  due to  \eqref{aux:bounded_away_01}, $H(\cdot)$ is $\mu$-integrable. Thus, by the generalized NP lemma \cite[Theorem 3.6.1 (iv)]{romano2005testing}, for $\mu$-a.e.~$y$, $\delta^*$ must take the following form:
\begin{align*}
    \delta^*(y) = \begin{cases}
      1, \;&\text{ if }  \;H(y)  > \kappa f_{\theta_0}(y) \\
      \gamma(y), \;&\text{ if }  \;H(y)  = \kappa f_{\theta_0}(y)\\
      0, \;&\text{ if }  \;H(y)  > \kappa f_{\theta_0}(y)
  \end{cases},
\end{align*}
for  some  $\kappa \in \bR$ and a measurable function $\gamma:\bY \to [0,1]$. The proof is complete for i) and ii).

Finally, we focus on iii). Define $\bar{\delta}_{\alpha}(y) = \alpha$ for $y \in \bY$; that is, the rule $\bar{\delta}_{\alpha} \in \Delta$ rejects $\mathcal{H}_0$ with probability $\alpha$ regardless of $y$. Clearly, $\bar{\delta}_{\alpha}$ is a feasible solution to problem \eqref{def:two_sided_prob}, and since $g(\cdot)$ is strictly increasing (Assumption \ref{assumption_g_increasing}),  the objective value for $\bar{\delta}_{\alpha}$ is strictly larger than that for $\bar{\delta}_{0}$. Thus, $\bar{\delta}_{0}$ is not an optimal solution to problem \eqref{def:two_sided_prob}. 

Let $\delta^*$ be an optimal solution to problem \eqref{def:two_sided_prob}, whose existence is guaranteed by i). Assume the contrary that $c^* := p(\theta_0; \delta^*) < \alpha$. Define a new rule as follows:
$$
\tilde{\delta}^*(y) = \delta^*(y) + \frac{\alpha - c^*}{1-c^*}(1-\delta^*(y)), \;\text{ for } y \in \bY.
$$
Clearly, $\tilde{\delta}^* \in \Delta$ is feasible for problem \eqref{def:two_sided_prob}, and since $g(\cdot)$ is strictly increasing, $\tilde{\delta}^*$ has a strictly larger objective value than $\delta^*$, which is a contradiction. The proof is complete.
\end{proof}
 
\section{Proofs of Corollary  \ref{cor:convexity} and  \ref{cor:exponential_family}}
\label{app:proof_corollaries}

\begin{proof}[Proof of Corollary \ref{cor:convexity}]
Let $\delta^*$ be an optimal rule in part i) of Theorem \ref{theorem:simple_null_structrual} and $p^*$ be the associated power function, that is, $p^*(\theta) = p(\theta;\delta^*)$ for $\theta \in \Theta$. Recall the form of $\delta^*$ in \eqref{def: optimal_rule_simple_null} and the definition of $H(\cdot)$ in \eqref{def:h_function}. By iii) of Theorem \ref{theorem:simple_null_structrual}, $p(\theta_0;\delta^*) = \alpha$.

For each fixed $\theta \in \Theta$, the following function is \emph{strictly} convex:
$$
T(y) \in \mathbbm{R} \mapsto  \phi_{\theta}(T(y))\in [0,\infty).
$$
Since $g'(t) > 0$ for $t \in (0,1)$  (Assumption \ref{assumption_g_increasing}) and $\Lambda_1$ is a probability measure, the following function
$$
T(y) \in \mathbbm{R} \mapsto \int_{\Theta_1}  \phi_{\theta}(T(y)) \nu(d\theta) \in [0,\infty),
$$
is also strictly convex, where $\nu$ is a \textit{measure} on $\Theta_1$ such that $d\nu/d\Lambda_1(\theta) = g'(p^*(\theta))$ for $\theta \in \Theta_1$. As a result, there exists a strictly convex function $\Psi:\bR \to \bR$ such that
$$
H(y)/f_{\theta_0}(y) = \Psi(T(y)),\;\text{ for } y \in \bY.
$$
Due to convexity, the following two sets are convex:
$\{T(y) \in \bR: \Psi(T(y)) < \kappa\}$, and
$\{T(y) \in \bR: \Psi(T(y)) \leq \kappa\}$,  which implies that
they are intervals in $\bR$. Further, since $\Psi(\cdot)$ is strictly convex, the set 
$\{T(y) \in \bR: \Psi(T(y)) = \kappa\}$ contains at most two points.
As a result,  there exist $-\infty \leq \ell < u \leq \infty$ such that
\begin{align*}
    (\ell,u) \subset &\{T(y) \in \bR: \Psi(T(y)) < \kappa\} \\
    \subset &\{T(y) \in \bR: \Psi(T(y)) \leq \kappa\} \subset [\ell,u].
\end{align*}
Since $\gamma(\cdot)$ is allowed to assume values in  $\{0,1\}$, the first claim follows immediately from Theorem \ref{theorem:simple_null_structrual}. The second claim is because $\mu(\{y \in \bY: T(y) \in \{\ell,u\}) = 0$ due to the assumption.
\end{proof}

\begin{proof}[Proof of Corollary \ref{cor:exponential_family}]
For a proof that Assumption \ref{assumption:continuity} holds, see Remark \ref{remark:1_2}. 
    Note that for $y \in \bY$,
    $$
    \frac{f_{\theta}(y)}{f_{\theta_0}(y)} = \phi_{\theta}(T(y)), \text{ for } y \in \mathbbm{Y},
    $$
    where $\phi_{\theta}(z):= \frac{c(\theta)}{c(\theta_0)} e^{(\theta-\theta_0)z}$ for $z \in \bR$. Thus for each $\theta \in \Theta_1$, $\phi_{\theta}$ is a strictly convex function on $\bR$. Thus condition \eqref{assumption:convex_suff} holds.
\end{proof}

\section{Proofs of Theorem \ref{thrm:composite_null_avg} and \ref{theorem:composiste_null_exponential} }\label{app:proof_composite_null}

\begin{proof}[Proof of Theorem \ref{thrm:composite_null_avg}] 
Recall the definition of $\tilde{f}_{\Theta_0}$ proceeding Theorem \ref{thrm:composite_null_avg}. By Fubini's theorem,
\begin{align*}
    \int_{\Theta_0} p(\theta;\delta) \Lambda_0(d\theta)
    =&\int_{\Theta_0} \left(\int_{\bY} \delta(y) f_{\theta}(y) \mu(dy) \right) \Lambda_0(d\theta)\\
=& \int_{\bY} \delta(y) \left(\int_{\Theta_0} f_{\theta}(y) \Lambda_0(d\theta)  \right)\mu(dy) \\
= &\int_{\bY} \delta(y) \tilde{f}_{\Theta_0}(y) \mu(dy).
\end{align*}
Thus, the optimization problem  in \eqref{def:two_sided_prob_composite} is equivalent to the following:
\begin{align*}
    \begin{split}
    &\max_{\delta \in \Delta} \int_{\Theta_1} g(p(\theta;\delta)) \Lambda_1(d\theta), \\ 
&\text{subject to } \int_{\bY} \delta(y) \tilde{f}_{\Theta_0}(y) \mu(dy) \leq \alpha.
\end{split}
\end{align*}
The proof then follows by applying nearly identical arguments as in the proof of Theorem \ref{theorem:simple_null_structrual}, replacing
$f_{\theta_0}$ with $\tilde{f}_{\Theta_0}$.
\end{proof}

Before proving Theorem \ref{theorem:composiste_null_exponential}, we start with supporting lemmas.

\begin{lemma}\label{lemma:compare_variance}
Let $a < b$ be two real numbers. Let $U$ be a random variable taking values in $[a,b]$, and $V$ be a random variable taking values in $(-\infty, a] \cup [b,\infty)$ such that $\Exp[V^2] < \infty$ and $\bP(V < a) + \bP(V > b) > 0$. If $\Exp[U] = \Exp[V] \in (a,b)$, then $\Var(U) < \Var(V)$.
\end{lemma}
\begin{proof}
Without loss of generality, assume $a = 0$ and $b=1$. Further, it suffices to show that $\Exp[U^2] < \Exp[V^2]$. 

Denote by $t = \Exp[U] \in (0,1)$ and since $U^2 \leq U$, we have
$\Exp[U^2] \leq \Exp[U] = t$.

Further, denote by $q_{1} = \bP\left(V \geq 1 \right)$ and 
 $q_{0} = \bP\left(V \leq 0\right)$. Since $\Exp[V] \in (0,1)$, we have
 $q_{0},q_{1} > 0$. Define $t_{1} = \Exp\left(V \vert V \geq 1 \right)$ and 
 $t_{0} = \Exp\left(V \vert V \leq 0 \right)$. By definition, $t_{1} \geq 1$, $t_0 \leq 0$, and $q_0 t_0 + q_1 t_1 = t$. Further, since $\bP(V < 0) + \bP(V > 1) > 0$, at least one of the following holds:
 $t_1 > 1$ or $t_0 < 0$.
  
 By Jensen's inequality,
 \begin{align*}
\Exp[V^2] &= q_1 \Exp[V^2 \vert V\geq 1] + q_0 \Exp[V^2|V \leq 0] \\
&\geq q_1 t_1^2 + q_0 t_0^2 > q_1 t_1 \geq t.
 \end{align*} 
 The proof is complete.
\end{proof}

\begin{lemma}\label{lemma:roots}
Let $f: \bR \to \bR$ be an analytic function, which is not identically zero. Assume that for any $x \in \bR$, if $f(x) = 0$, we have $f'(x) > 0$. Then $f$ has at most one root.
\end{lemma}
\begin{proof}
Without loss of generality, assume $f$ has a root at $a$.    Assume that $f$ has another root on $(a,\infty)$. Since $f$ is a non-zero analytic function, there exists a smallest root on $(a,\infty)$, denoted by $b$. Since $f(a) = 0$, which implies that $f'(a) > 0$, and due to the minimality of $b$, $f(x) > 0$ for $x \in (a,b)$. This contradicts with the fact that $f(b) = 0$, but $f'(b) > 0$.

By a similar argument, there can be no root on $(-\infty,a)$. The proof is complete.
\end{proof}

The following change-of-measure result (also known as exponential tilting) is well-known in the literature and is provided here for completeness.

\begin{lemma}\label{lemma:change_of_measure}
Let $U$ be a bounded random variable with distribution $F$. For each $x \in \bR$, define a new distribution:
$$
\frac{d F_x}{d F} (u) = \frac{e^{x u}}{\int e^{x u} F(du)}, \text{ for } u \in \bR.
$$
Let $U^{(x)}$ be a random variable with distribution $F_x$. Then $U^{(x)}$ has the same support as $U$, and
\begin{align*}
&\Exp[U^{(x)}] = \left(\log\Exp\left[e^{xU}\right]\right)', \\
&\Var(U^{(x)}) = \left(\log\Exp\left[e^{xU}\right]\right)''.
\end{align*}
\end{lemma}

Now we prove Theorem \ref{theorem:composiste_null_exponential}.
\begin{proof}[Proof of Theorem \ref{theorem:composiste_null_exponential}]
Recall $H(\cdot)$ in \eqref{def:h_function} and $\tilde{f}_{\Theta_0}(\cdot)$ proceeding Theorem \ref{thrm:composite_null_avg}. Define 
$$
\widetilde{H}(y) := \frac{H(y)}{\tilde{f}_{\Theta_0}(y)}, \; \text{ for } y \in \bY.
$$
By the definition of $f_{\theta}$ in \eqref{def:exponential_family}, we have
\begin{align*}
    \widetilde{H}(y) = \frac{\int_{\Theta_1} g'(p(\theta;\delta^*)) c(\theta) e^{\theta T(y)}  \Lambda_1(d \theta) }{\int_{\Theta_0} c(\theta) e^{\theta T(y)}\Lambda_0(d \theta)}.
\end{align*}
By Theorem \ref{thrm:composite_null_avg}, for any optimal rule $\delta^* \in \Delta$, it rejects $\mathcal{H}_0$ with probability $1$, $\gamma(y)$, $0$ if $\widetilde{H}(y) >, =, < \kappa$, respectively, for some  $\kappa \geq 0$ and $\gamma:\bY \to [0,1]$.

Note that $g'(t) > 0$ for $t \in (0,1)$ (Assumption \ref{assumption_g_increasing}). 
Let $U$ and $V$ be two random variables with the distributions $F_U$ and $F_V$, respectively, where
\begin{align*}
&\frac{d F_U}{d \Lambda_1}(\theta) = \frac{1}{K_1} g'(p(\theta;\delta^*) c(\theta),\; \text{ for } \theta \in \Theta_1\\
&\frac{d F_V}{d \Lambda_0}(\theta) = \frac{1}{K_2} c(\theta), \; \text{ for } \theta \in \Theta_0,
\end{align*}
and the normalizing constants are defined as
\begin{align*}
    &K_1 := \int_{\Theta_1} g'(p(\tilde{\theta};\delta^*) c(\tilde{\theta})  \Lambda_1(d \tilde{\theta}), \\
    &K_2 := \int_{\Theta_0}  c(\tilde{\theta})  \Lambda_0(d \tilde{\theta}).
\end{align*}
By definition, we have
\begin{align*}
     \widetilde{H}(y) := \frac{K_1}{K_2} \frac{\Exp[e^{U T(y)}]}{\Exp[e^{V T(y)}]}.
\end{align*}
Thus $\widetilde{H}(y) >, =, < \kappa$ is equivalent to, respectively, the following:
$$
\log(\Exp[e^{U T(y)}]) - \log(\Exp[e^{V T(y)}]) >,=,< \tilde{\kappa},
$$
where $\tilde{\kappa} := \log(\kappa) - \log(K_1/K_2)$. 
Define for $x \in \bR$,
$$
\psi(x) :=\log(\Exp[e^{U x}]) - \log(\Exp[e^{V x}]) - \tilde{\kappa}.
$$
Recall the definition of $U^{(x)}$ in Lemma \ref{lemma:change_of_measure}, and we define $V^{(x)}$ similarly. Then by Lemma \ref{lemma:change_of_measure},
\begin{align*}
&\psi'(x) = \Exp\left[U^{(x)}\right] - \Exp\left[V^{(x)}\right], \\
&\psi''(x) = \Var\left[U^{(x)}\right] - \Var\left[V^{(x)}\right].
\end{align*}
Since for each $x \in \bR$, $U^{(x)}$ is supported on $\Theta_1 = [-K,a) \cup (b,K]$, and $V^{(x)}$ on $\Theta_0 = [a,b]$. By Lemma \ref{lemma:compare_variance}, if $\psi'(x) = 0$ for some $x \in \bR$, then $\psi''(x) > 0$. Then, by Lemma \ref{lemma:roots}, $\psi'$ has at most one root on $\bR$.

If $\psi'$ has no root, then $\psi(T(y))$ is a monotone function of $T(y)$. If $\psi'$ has exactly one root, then $\psi(T(y))$ is either ``first increasing, then decreasing'' or ``first decreasing, then increasing'' with $T(y)$. Since $\Theta_0 = [a,b]$, as $|T(y)| \to \infty$, $\psi(T(y)) \to \infty$, which implies that $\delta^*$ must be of the form given in \eqref{def:optimal_rule_Ty}.
\end{proof}

\section{Proofs of Theorem \ref{theorem:composiste_supremum} and Theorem   \ref{theorem:composiste_exponential_supremum}}\label{app:composite_supremum}
\begin{proof}[Proof of Theorem \ref{theorem:composiste_supremum}]
By arguments similar to those in Theorem \ref{theorem:simple_null_structrual}, optimal decision rules exist for \eqref{def:composite_null_sup}. Let $\delta^* \in \Delta$ be an optimal rule for \eqref{def:composite_null_sup} and $p^* \in \cP$ be the associated power function. 

Again, by arguments similar to those in Theorem \ref{theorem:simple_null_structrual}, we have that 
$\sup_{\theta \in \Theta_0} p(\theta;\delta^*) = \alpha$ and
\eqref{aux:bounded_away_01} holds for $p^*$, and that $\delta^*$ is the optimal solution to the following optimization problem:
\begin{align*}
      \begin{split}
&\max_{\delta \in \Delta}   \int_{\bY} {\delta}(y) H(y) \mu(dy), \\
&\text{subject to } \sup_{\theta \in \Theta_0} p(\theta; \delta) \leq \alpha,
\end{split}  
\end{align*}
where recall that $H(\cdot)$ is defined in \eqref{def:h_function}. Since $g'(t) > 0$ for $0 < t < 1$ (Assumption \ref{assumption_g_increasing}), due to \eqref{aux:bounded_away_01}, we have  $H(y) \geq 0$ for $y \in \bY$ and $0 < \int_{\bY} H(y) \mu(dy) < 1$. 

Since $\Theta_0$ is closed and $\Theta$ is compact, by  Theorem 3.8.1 of \cite{romano2005testing} (see also the remark on Page 86 of \cite{romano2005testing} and \cite{lehmann1952existence}), there exists a probability measure $\widetilde{\Lambda}_0$ on $\Theta_0$ such that $\delta^*$ is the optimal solution to the following optimization problem:
\begin{align*}
      \begin{split}
&\max_{\delta \in \Delta}   \int_{\bY} {\delta}(y) H(y) \mu(dy), \\
&\text{subject to } \int_{\Theta_0}  \delta(y) \widetilde{f}_{\widetilde{\Lambda}_0}(y) \mu(dy) \leq \alpha,
\end{split}  
\end{align*}
where $ \widetilde{f}_{\widetilde{\Lambda}_0}$ is defined in Theorem \ref{theorem:composiste_supremum}. Then the proof is complete by the NP lemma \cite[Theorem 3.6.1]{romano2005testing},
\end{proof}

Next, we prove Theorem   \ref{theorem:composiste_exponential_supremum}.

\begin{proof}[Proof of Theorem \ref{theorem:composiste_exponential_supremum}]
Let $\delta^*$ be the optimal rule in Theorem \ref{theorem:composiste_supremum}. 
Recall $H(\cdot)$ in \eqref{def:h_function}, $\widetilde{f}_{\widetilde{\Lambda}_0}(\cdot)$ in \ref{theorem:composiste_supremum} and  $f_{\theta}$ in \eqref{def:exponential_family}. Define for $y \in \bY$,
$$
\bar{H}(y):= \frac{H(y)}{\widetilde{f}_{\widetilde{\Lambda}_0}(y)} =\frac{\int_{\Theta_1} g'(p(\theta;\delta^*)) c(\theta) e^{\theta T(y)}  \Lambda_1(d \theta) }{\int_{\Theta_0} c(\theta) e^{\theta T(y)}\widetilde{\Lambda}_0(d \theta)}.
$$
Thus for $\mu$-a.e.~$y$, $\delta^*(y) = 1, \gamma(y), 0$ if 
$\bar{H}(y) >, =, < 1$ respectively. In the proof of Theorem \ref{theorem:composiste_null_exponential}, it is shown that $\delta^*$ must be of the form given in \eqref{def:optimal_rule_Ty}.

It remains to show that $\max\{p(a;\delta^*), p(b;\delta^*)\} = \alpha$. Assume the contrary that 
$\max\{p(a;\delta^*), p(b;\delta^*)\} = \alpha' < \alpha$. 

Denote by $\bar{\theta}$ a maximizer of the function $p(\cdot;\delta^*)$ on $[a,b]$. Since $\alpha' < \alpha = \sup_{\theta \in [a,b]} p(\theta;\delta^*)$, we must have $\bar{\theta} \in (a,b)$ and $p(\bar{\theta};\delta^*) > \alpha$.  Further,
\begin{align}
    \label{aux:theta_bar}
    \frac{d}{d\theta} \log p(\bar{\theta}; \delta^*) = 0, \;\;
    \frac{d^2}{d^2\theta} \log p(\bar{\theta}; \delta^*) \leq 0.
\end{align}

For each $\theta \in [-K,K]$, recall the definition of $f_{\theta}$   in \eqref{def:exponential_family}, and define two probabilities measures $\mu_{\theta,+}$ and $ \mu_{\theta,-}$ on $\bY$ as follows: for each $y \in \bY$, 
\begin{align*}
    &d \mu_{\theta,+}/d\mu(y) = (K_{\theta,+})^{-1}\delta^*(y) h(y) e^{{\theta} T(y)}, \\
    &d \mu_{\theta,-}/d\mu(y) = (K_{\theta,-})^{-1}\bar{\delta}^*(y) h(y) e^{{\theta} T(y)}, 
\end{align*}
where  $\bar{\delta}^*(y) := 1 - \delta^*(y)$ and
\begin{align*}
K_{\theta,+} := \int_{\bY} \delta^*(\tilde{y}) h(\tilde{y}) e^{{\theta} T(\tilde{y})} \mu(d\tilde{y}), \\
K_{\theta,-} := \int_{\bY} \bar{\delta}^*(\tilde{y}) h(\tilde{y}) e^{{\theta} T(\tilde{y})} \mu(d\tilde{y}).
\end{align*}
By definition, $p(\theta;\delta^*) = K_{\theta,+}/(K_{\theta,+}+K_{\theta,-})$, and 
for $\iota \in \{+,-\}$,
\begin{align*}
    &\frac{d}{d\theta} K_{\theta, \iota} =K_{\theta,\iota} \int_{\bY} T(y) \mu_{\theta,\iota}(dy), \\
    &\frac{d^2}{d^2\theta} K_{\theta, \iota} =K_{\theta,\iota} \int_{\bY} (T(y))^2 \mu_{\theta,\iota}(dy). 
\end{align*}
By elementary calculation of the first two derivatives of $\log p({\theta};\delta^*)$ at $\bar{\theta}$, and due to \eqref{aux:theta_bar},  we have
\begin{align*}
    &\int_{\bY} T(y)\mu_{\bar{\theta},+}(dy) = \int_{\bY} T(y)\mu_{\bar{\theta},-}(dy),\\
    &\int_{\bY} (T(y))^2 \mu_{\bar{\theta},+}(dy) \leq  \int_{\bY} (T(y))^2 \mu_{\bar{\theta},-}(dy).
\end{align*}
Let $Y_{+}$ and $Y_{-}$ be two random variables with distribution $\mu_{\bar{\theta},+}$ and $\mu_{\bar{\theta},+}$ respectively. Then, the above-displayed equation implies that 
$$\Exp[T(Y_{+})] = \Exp[T(Y_{-})],\;\;\text{Var}(T(Y_{+})) \leq \text{Var}(T(Y_{-})).
$$
 However,  by the definition of $\delta^*$, $T(Y_{+})$ is supported on $ (-\infty,\ell] \cup [u, \infty)$, and $T(Y_{-})$ on $ [\ell,u]$. Further, since $T(Y_{+})$ and $T(Y_{-})$ take at least two values, we have that
 $\bP(T(Y_{+}) < \ell) + \bP(T(Y_{+}) > u) > 0$,
 and  $\bP(\ell< T(Y_{-}) < u) > 0$. Then by Lemma \ref{lemma:compare_variance}, if $\Exp[T(Y_{+})] = \Exp[T(Y_{-})] \in (\ell,u)$, we must have $\text{Var}(T(Y_{+})) > \text{Var}(T(Y_{-}))$, which is a contradiction. The proof is complete.
\end{proof}

\bibliographystyle{IEEEtran}
\bibliography{reference}

\end{document}